\documentclass[11pt,a4paper]{article}

\usepackage{stmaryrd}
\usepackage{enumerate}
\usepackage{hyperref}
\usepackage{amsthm}

\usepackage{color}
\usepackage{lineno}
\usepackage{graphicx}
\usepackage{ae}
\usepackage{amsmath}
\usepackage{amssymb}
\usepackage{latexsym}
\usepackage{url}
\usepackage{epsfig}
\usepackage{mathrsfs}
\usepackage{amsfonts}
\usepackage{amsthm}

\newcommand{\pmat}[1]{\begin{pmatrix}#1\end{pmatrix}}

\newcommand{\NEPS}{\mathrm{NEPS}}

\def\mod{{\rm mod}}

\newtheorem*{theorem*}{Theorem}

\newtheorem{theorem}{Theorem}[section]

\newtheorem{lemma}[theorem]{Lemma}

\newtheorem{cor}[theorem]{Corollary}

\newtheorem*{assumption*}{Assumption}

\newtheorem{example}{Example}

\theoremstyle{definition}
\newtheorem{definition}{Definition}[section]

\newtheorem{problem}{Problem}

\numberwithin{equation}{section} 
\allowdisplaybreaks  

\def\qed{\hfill$\Box$\vspace{12pt}}

\long\def\delete#1{}

\usepackage{xcolor}
\usepackage[normalem]{ulem}

\begin{document}
\title{Perfect state transfer in NEPS of complete graphs}

\date{}
\author{Yipeng Li$^{a,b}$,
Xiaogang Liu$^{a,c,}$\thanks{Supported by the National Natural Science Foundation of China (Nos. 11601431 and 11871398), the Natural Science Foundation of Shaanxi Province (No. 2020JM-099), the Natural Science Foundation of Qinghai Province  (No. 2020-ZJ-920).}~$^,$\thanks{Corresponding author. Email addresses: pengyl@xust.edu.cn, xiaogliu@nwpu.edu.cn, sgzhang@nwpu.edu.cn, sanming@unimelb.edu.au }~,~
Shenggui Zhang$^{a,c,}$\thanks{Supported by the National Natural Science Foundation of China (Nos.~11571135 and 11671320) and the Fundamental Research Funds for the Central Universities (No. 3102019GHJD003).}~, Sanming Zhou$^{d}$ \\[2mm]
\small $^a$School of Mathematics and Statistics, Northwestern Polytechnical University,\\[-0.8ex]
\small Xi'an, Shaanxi 710072, P.R.~China\\
\small $^b$Department of Applied Mathematics, Xi'an University of Science and Technology,\\[-0.8ex]
\small Xi'an, Shaanxi 710054, P.R.~China\\
{\small $^c$Xi'an-Budapest Joint Research Center for Combinatorics,}\\[-0.8ex]
{\small Northwestern Polytechnical University, Xi'an, Shaanxi 710129, P.R. China}\\
\small $^d$School of Mathematics and Statistics, The University of Melbourne, \\[-0.8ex]
\small Parkville, VIC 3010, Australia
}

\openup 0.5\jot
\maketitle

\begin{abstract}
Perfect state transfer in graphs is a concept arising from quantum physics and quantum computing. Given a graph $G$ with adjacency matrix $A_G$, the transition matrix of $G$ with respect to $A_G$ is defined as $H_{A_{G}}(t) = \exp(-\mathrm{i}tA_{G})$, $t \in \mathbb{R},\ \mathrm{i}=\sqrt{-1}$. We say that perfect state transfer from vertex $u$ to vertex $v$ occurs in $G$ at time $\tau$ if $u \ne v$ and the modulus of the $(u,v)$-entry of $H_{A_G}(\tau)$ is equal to $1$. If the moduli of all diagonal entries of $H_{A_G}(\tau)$ are equal to $1$ for some $\tau$, then $G$ is called periodic with period $\tau$. In this paper we give a few sufficient conditions for NEPS of complete graphs to be periodic or exhibit perfect state transfer.

\smallskip


\end{abstract}

\section{Introduction}

Let $G$ be a graph. Denote by $V(G)$ the vertex set of $G$ and $A_G$ the adjacency matrix of $G$. The \emph{transition matrix} of $G$ with respect to $A_G$ is defined as
$$
H_{A_{G}}(t) = \exp(-\mathrm{i}tA_{G})=\sum_{k\ge0}\frac{(-\mathrm{i})^{k} A^{k}_{G} t^{k}}{k!}, ~ t \in \mathbb{R},~\mathrm{i}=\sqrt{-1}.
$$
Clearly, this is a symmetric and unitary matrix for any $t$. Denote by $H_{A_G}(t)_{u,v}$ the $(u,v)$-entry of $H_{A_{G}}(t)$, where $u, v \in V(G)$. We say that \emph{perfect state transfer} (\emph{PST} for short) from vertex $u$ to vertex $v$ occurs in $G$ at time $\tau$ if $u \ne v$ and $|H_{A_G}(\tau)_{u,v}|=1$. If $|H_{A_G}(\tau)_{u,u}|=1$, then $G$ is called \emph{periodic} at vertex $u$ with period $\tau$. A graph is called \emph{periodic} if it is periodic at every vertex with the same period.

The definitions above arose from quantum physics (see, for example, \cite{SBose, chris1}).
In recent years, the problem of determining which quantum spin networks admit PST has received considerable attention due to its potential applications in quantum information transmission and quantum computing. It is known \cite{FU} that PST can be used as an important approach to universal quantum computing. In \cite{SBose}, Bose first introduced the concept of PST and addressed its importance in quantum computing. In \cite{chris1}, Chirstandl et al. modelled quantum spin networks by graphs in which vertices represent locations of the qubits and edges represent quantum wires between such qubits. In this way the problem of determining which quantum spin networks admit PST can be transformed to the one of characterizing which graphs admit PST. Unfortunately, in general it is difficult to determine whether a given graph exhibits PST. Up until now, only a small number of families of graphs have been proved to admit PST. See \cite{Basic11, SBose, Coutinho15, Godsil12, HPal, Zhou14, Zhou15, SZ} and the survey papers \cite{CGodsil, Stevanovic11} for details.

The purpose of this paper is to study the existence of PST and periodicity in the family of graphs built from complete graphs using the $\NEPS$ operation. As usual we use $K_n$, $P_n$ and $C_n$ to denote the complete graph, the path and the cycle on $n$ vertices, respectively. Denote the elementary abelian $2$-group of rank $d \ge 1$ by $\mathbb{Z}_2^{d}$ and its identity element by $\mathbf{0}=(0,\ldots,0)$, where $\mathbb{Z}_2 = \{0,1\}$ is the group of integers modulo $2$. We will omit the rank $d$ in $\mathbf{0}$ as it can be easily figured out in the context.

\begin{definition}
(Cvetkovi\'{c} et al. \cite[Definition 2.5.1]{SG1})
\label{DefNEPS1}
Let $G_1, \ldots,G_d$ be graphs, and let $\emptyset\neq \mathcal{A}\subseteq \mathbb{Z}_2^{d}\setminus \{\mathbf{0}\}$, where $d \ge 1$ and $\mathbf{0}=(0,\ldots,0) \in \mathbb{Z}_2^{d}$. The NEPS (non-complete extended $p$-sum) of $G_1,\ldots,G_d$ with \emph{basis} $\mathcal{A}$, denoted by $\NEPS(G_1,\ldots,G_d;\mathcal{A})$, is the graph with vertex set $V(G_{1})\times \cdots \times V(G_{d})$ in which two vertices $(u_{1}, \ldots, u_{d})$ and $(v_{1}, \ldots, v_{d})$ are adjacent if and only if there exists
$(a_{1}, \ldots, a_{d})\in \mathcal{A}$ such that $u_{i}=v_{i}$ whenever $a_{i}=0$ and $u_{i}$ is adjacent to $v_{i}$ in $G_{i}$ whenever $a_{i}=1$.
\end{definition}

The notion of NEPS is a generalization of several graph operations such as tensor product (also known as direct product, Kronecker product, categorical product, etc. in the literature \cite{ImrichK00}), Cartesian product, and strong product. In fact, $\NEPS(G_1,\ldots,G_d; \{(1, \ldots, 1)\})$ is simply the \emph{tensor product} $G_1 \otimes \cdots \otimes G_d$. If $\mathcal{A} = \{(1, 0, \ldots, 0), \ldots, (0, 0, \ldots, 1)\}$ is the standard basis of $\mathbb{Z}_2^{d}$, then $\NEPS(G_1,\ldots,G_d;\mathcal{A})$ is the \emph{Cartesian product} $G_1 \Box \cdots \Box G_d$; in particular, $\NEPS(K_{n_1},\ldots,K_{n_d}; \mathcal{A})$ is the \emph{Hamming graph} $H(n_1, \ldots, n_d)$, where $n_1, \ldots, n_d \ge 2$, and $\NEPS(K_{2},\ldots,K_{2}; \mathcal{A})$ is the hypercube $Q_d$ of dimension $d$. In general, for any $\emptyset \neq \mathcal{A}\subseteq \mathbb{Z}_2^{d} \setminus \{\mathbf{0}\}$, $\NEPS(K_{2},\ldots,K_{2}; \mathcal{A})$ is called a \emph{cubelike graph}.

The $\NEPS$ operation provides a useful tool for constructing graphs admitting PST. For example, Pal and Bhattacharjya \cite{HPal1} gave sufficient conditions for the $\NEPS$ of some copies of $P_3$ to admit PST, and Zheng et al. \cite{SZ} obtained sufficient conditions for the $\NEPS$ of some copies of cubes to admit PST. It was proved in \cite[Theorem 1]{AB} and \cite[Theorem 2.3]{CC} that a cubelike graph $\NEPS(K_{2},\ldots,K_{2}; \mathcal{A})$ exhibits PST if $\sum_{\mathbf{a} \in \mathcal{A}} \mathbf{a} \ne \mathbf{0}$, where the sum on the left-hand side is performed in $\mathbb{Z}_2^{d}$, with each coordinate modulo $2$. On the other hand, it is known \cite[Corollary 2]{SZ} that any NEPS of complete graphs $K_{n_1},\ldots,K_{n_{d}}$ with $n_{i} \geq 3$ for each $i$ cannot exhibit PST. In the general case when $n_{i} \ge 3$ for at least one $i$ and $n_{j} = 2$ for at least one $j$, it is unknown whether $\NEPS(K_{n_1},\ldots, K_{n_d}; \mathcal{A})$ admits PST. We aim to study this problem and the related periodicity problem in this paper. As we will see shortly, our results extend the known results \cite{AB,CC} on cubelike graphs and enlarge the collection of graphs that admit PST. In addition, the methods developed in this paper can be used to give another proof of these results \cite{AB,CC} on cubelike graphs.

In the sequel we always assume that $d, r, n_1, \ldots, n_d$ are integers such that
$$
d \ge 1,\ r \ge 1,\ n_1, \ldots, n_d \ge 3.
$$
The greatest common divisor of $n_1,\ldots,n_d$ is denoted by $\gcd(n_1,\ldots,n_d)$.
Write $r\odot K_2 = \underbrace{K_2,\ldots, K_2}_{r}$, and for $\emptyset\neq \mathcal{A}\subseteq \mathbb{Z}_2^{d+r}\setminus \{\mathbf{0}\}$, set
$$
\NEPS(K_{n_1},\ldots,K_{n_{d}},r\odot K_2;\mathcal{A}) = \NEPS(K_{n_1},\ldots,K_{n_{d}}, \underbrace{K_2,\ldots, K_2}_{r}; \mathcal{A}).
$$
As usual, the \emph{Hamming weight} of a vector $\mathbf{a} = (a_{1}, \ldots, a_{m}) \in \mathbb{Z}_2^m$, where $m \ge 1$, is defined as
$$
w(\mathbf{a}) = w(a_{1}, \ldots, a_{m}) = |\{i: 1 \le i \le m,\ a_i = 1\}|.
$$
We make the following assumption throughout the paper.

\begin{assumption*}
{\em
For $\mathcal{A} \subseteq \mathbb{Z}_2^{m}$, where $m \ge 1$, we set
$$
\mathbf{c}(\mathcal{A}) = \sum_{\mathbf{a} \in \mathcal{A}} \mathbf{a},
$$
where the addition is performed in $\mathbb{Z}_2^{m}$ (that is, each coordinate takes modulo $2$). We treat $V(K_{n_i})$ as the cyclic group $\mathbb{Z}_{n_i}$ for $1 \le i \le d$ and treat $V(K_{2})$ as $\mathbb{Z}_{2}$. In this way each vertex of $\NEPS(K_{n_1},\ldots,K_{n_{d}},r\odot K_2;\mathcal{A})$ is treated as an element $\mathbf{u} = (u_1, \ldots, u_d, u_{d+1}, \ldots, u_{d+r}) \in \mathbb{Z}_{n_1} \times \cdots \times \mathbb{Z}_{n_d} \times \mathbb{Z}_{2}^r$. Moreover, for any $\mathbf{a} = (a_1, \ldots, a_d, a_{d+1}, \ldots, a_{d+r}) \in \mathbb{Z}_{2}^{d+r}$, $\mathbf{u} + \mathbf{a}$ is understood as the element of $\mathbb{Z}_{n_1} \times \cdots \times \mathbb{Z}_{n_d} \times \mathbb{Z}_{2}^r$ obtained by coordinate-wise addition, with the $i$th coordinate modulo $n_i$ for $1 \le i \le d$ and the $j$th coordinate modulo $2$ for $d+1 \le j \le d+r$. That is,
$$
\mathbf{u} + \mathbf{a} = ((u_1 + a_1)\ \mod\ n_1, \ldots, (u_d + a_d) \ \mod\ n_d, (u_{d+1} + a_{d+1}) \ \mod\ 2, \ldots, (u_{d+r} + a_{d+r}) \ \mod\ 2).
$$}
\end{assumption*}

Using the notation and assumption above, we now present our results in the following three theorems.

\begin{theorem}
\label{Theorem3}
Let $\emptyset \neq\mathcal{A} \subseteq \mathbb{Z}_2^{d+r}\setminus \{\mathbf{0}\}$. Suppose that the last $r$ coordinates of each element of $\mathcal{A}$ are equal to $0$. Then $\NEPS(K_{n_1},\ldots,K_{n_d},r\odot K_2;\mathcal{A})$ is periodic with period $\frac{2\pi}{h}$, where $h=\gcd(n_1,\ldots,n_d)$.
\end{theorem}

\begin{theorem}
\label{Theorem4}
Let $\emptyset\neq\mathcal{A}\subseteq\mathbb{Z}_2^{d+r}\setminus \{\mathbf{0}\}$. Suppose that the first $d$ coordinates of each element of $\mathcal{A}$ are equal to $0$. Denote $\mathbf{c}(\mathcal{A}) = (0, \ldots, 0, c_{d+1}, \ldots, c_{d+r})$.
\begin{itemize}
\item[\rm (a)]  If $\mathbf{c}(\mathcal{A}) \neq \mathbf{0}$, then $\NEPS(K_{n_1},\ldots,K_{n_d},r\odot K_2;\mathcal{A})$ admits PST from vertex $\mathbf{u}$ to vertex $\mathbf{u}+\mathbf{c}(\mathcal{A})$ at time $\frac{\pi}{2}$, for every vertex $\mathbf{u}$, and its transition matrix $H_{\mathcal{A}}(t)$ at $\frac{\pi}{2}$ is given by
$$
H_{\mathcal{A}}\left(\frac{\pi}{2}\right)
=(-\mathrm{i})^{\mid \mathcal{A}\mid}I_{n_1\cdots n_d}\otimes\left(\bigotimes_{j=d+1}^{d+r}(A_{K_2})^{c_j}\right).
$$
 \item[\rm (b)]
If $\mathbf{c}(\mathcal{A}) = \mathbf{0}$, then $\NEPS(K_{n_1},\ldots,K_{n_d},r\odot K_2;\mathcal{A})$ is periodic with period $\frac{\pi}{2}$, and its transition matrix $H_{\mathcal{A}}(t)$ at $\frac{\pi}{2}$ is given by
$$
H_{\mathcal{A}}\left(\frac{\pi}{2}\right)
=(-\mathrm{i})^{\mid \mathcal{A}\mid}I_{2^rn_1\cdots n_d}.
$$
\end{itemize}
\end{theorem}

Note that $\NEPS(K_{n_1},\ldots,K_{n_d},r\odot K_2;\mathcal{A})$ in Theorem \ref{Theorem4} is the vertex-disjoint union of $n_1n_2\cdots n_r$ copies of the cubelike graph $\NEPS(r\odot K_2;\mathcal{A}_2)$, where $\mathcal{A}_2$ consists of those $(a_1,\ldots,a_r) \in \mathbb{Z}_2^{r}$ such that $(0,\ldots,0,a_1,\ldots,a_r)\in \mathcal{A}$.


For any $\mathcal{A}\subseteq \mathbb{Z}_2^{d+r}$, define
$$
\mathcal{A}^{*} =\{(a_1,\ldots,a_d): (a_1,\ldots,a_d, a_{d+1},\ldots,a_{d+r}) \in \mathcal{A}\}.
$$
For any $\mathbf{a} = (a_1,\ldots,a_d) \in \mathbb{Z}_2^{d}$, define
$$
\mathcal{A}_{-}{(\mathbf{a})} =\{(x_{d+1},\ldots,x_{d+r}): (x_1,\ldots,x_d, x_{d+1},\ldots,x_{d+r}) \in \mathcal{A} \text{ and } (x_1,\ldots,x_d) = \mathbf{a}\}.
$$
In particular,
$$
\mathcal{A}_{-}{(\mathbf{0})} = \{(a_{d+1},\ldots,a_{d+r}): (0,\ldots,0,a_{d+1},\ldots,a_{d+r})\in \mathcal{A}\}.
$$
Note that if $\mathbf{a} \not \in \mathcal{A}^{*}$ then $\mathcal{A}_{-}{(\mathbf{a})} = \emptyset$.

\begin{theorem}
\label{10-3}
Let $\emptyset\not=\mathcal{A}\subseteq \mathbb{Z}_2^{d+r}\setminus \{\mathbf{0}\}$. Suppose that $\mathcal{A}$ contains at least one element whose last $r$ coordinates are not all $0$ and at least one element whose first $d$ coordinates are not all $0$. Set $\mathbf{c}(\mathcal{A}) = (b_1, \ldots, b_d, b_{d+1}, \ldots,b_{d+r})$ and $\mathbf{c}(\mathcal{A}_{-}{(\mathbf{0})}) = (c_{d+1},\ldots,c_{d+r})$. Denote $G = \NEPS(K_{n_1},\ldots,K_{n_{d}},r\odot K_2;\mathcal{A})$ and $h=\gcd(n_1,\ldots,n_d)$.
\begin{itemize}
\item[\rm (a)]
Suppose that $h$ is odd. Then $G$ is periodic with period $2\pi$.
Moreover, if every $\mathbf{a} \in \mathcal{A}^{*}\setminus \{\mathbf{0}\}$ satisfies $|\mathcal{A}_{-}{(\mathbf{a})}| \equiv 0~(\mod~4)$ and $\mathbf{c}(\mathcal{A}_{-}{(\mathbf{a})}) = \mathbf{0}$, then the following hold:
\begin{itemize}
\item[\rm (a.i)] if $(c_{d+1},\ldots,c_{d+r}) \ne \mathbf{0}$, then $G$ admits PST from $\mathbf{u}$ to $\mathbf{u}+(0,\ldots,0,c_{d+1},\ldots,c_{d+r})$ at time $\frac{\pi}{2}$, for every vertex $\mathbf{u}$;
\item[\rm (a.ii)] if $(c_{d+1},\ldots,c_{d+r}) = \mathbf{0}$, then $G$ is periodic with period $\frac{\pi}{2}$.
\end{itemize}

\item[\rm (b)] Suppose that $h$ is even. Then $G$ is periodic with period $\pi$. Moreover, if $h$ is not a multiple of $4$ and every $\mathbf{a} \in \mathcal{A}^{*}\setminus \{\mathbf{0}\}$ satisfies $|\mathcal{A}_{-}{(\mathbf{a})}| \equiv 0~(\mod~2)$ and $\mathbf{c}(\mathcal{A}_{-}{(\mathbf{a})}) = \mathbf{0}$, then the following hold:
\begin{itemize}
\item[\rm (b.i)] if $(c_{d+1},\ldots,c_{d+r}) \ne \mathbf{0}$, then $G$ admits PST from $\mathbf{u}$ to $\mathbf{u}+(0,\ldots,0,c_{d+1},\ldots,c_{d+r})$ at time $\frac{\pi}{2}$, for every vertex $\mathbf{u}$;
\item[\rm (b.ii)] if $(c_{d+1},\ldots,c_{d+r}) = \mathbf{0}$, then $G$ is periodic with period $\frac{\pi}{2}$.
\end{itemize}

\item[\rm (c)] Suppose that $h$ is a multiple of $4$. Then the following hold:
\begin{itemize}
\item[\rm (c.i)]
if $(b_{d+1},\ldots,b_{d+r})\neq \mathbf{0}$, then $G$ admits PST from $\mathbf{u}$ to $\mathbf{u}+(0,\ldots,0,b_{d+1},\ldots,b_{d+r})$ at time $\frac{\pi}{2}$, for every vertex $\mathbf{u}$;
\item[\rm (c.ii)] if $(b_{d+1},\ldots,b_{d+r})=\mathbf{0}$, then $G$ is periodic with period $\frac{\pi}{2}$.
\end{itemize}
\end{itemize}
\end{theorem}


Recall that, for $\mathcal{A} = \{(1, 0, \ldots, 0), \ldots, (0, 0, \ldots, 1)\} \subset \mathbb{Z}_2^{d+r}\setminus \{\mathbf{0}\}$, $\NEPS(K_{n_1},\ldots,K_{n_{d}},r\odot K_2;\mathcal{A})$ is the Hamming graph $H(n_1, \ldots, n_d, 2,\ldots, 2)$, which is written as $H(n_1, \ldots, n_d, r \odot 2)$ in the sequel. Theorem \ref{10-3} implies the following corollary.

\begin{cor}
\label{corol:Hamming}
Let $h=\gcd(n_1,\ldots,n_d)$.
\begin{itemize}
\item[\rm (a)]
If $h$ is odd, then $H(n_1, \ldots, n_d, r \odot 2)$ is periodic with period $2\pi$.
\item[\rm (b)]
If $h$ is even, then $H(n_1, \ldots, n_d, r \odot 2)$ is periodic with period $\pi$.
\item[\rm (c)] If $h$ is a multiple of $4$, then $H(n_1, \ldots, n_d, r \odot 2)$ admits PST from vertex $\mathbf{u}$ to vertex $\mathbf{u}+(0,\ldots,0,1,\ldots,1)$ at time $\frac{\pi}{2}$, for every vertex $\mathbf{u}$.
\end{itemize}
\end{cor}

We will prove Theorems \ref{Theorem3}--\ref{10-3} in Section \ref{Sec:main} after preliminary discussions in Section \ref{Sec:pre}. 

As will be seen in Examples \ref{Exampe:2} and \ref{Exampe:3} in the last section, the sufficient condition in part (a) of Theorem \ref{Theorem4} is in general not necessary for $\NEPS$ of complete graphs to admit PST, and similarly the sufficient condition in part (c.i) of Theorem \ref{10-3} is in general not necessary. So it is natural to ask for which special families of $\NEPS$ of complete graphs the sufficient conditions in these parts are also necessary. More broadly, we would like to pose the following general problem.
\begin{problem}
Give a necessary and sufficient condition for $\NEPS$ of complete graphs to admit PST.
\end{problem}


\section{Preliminaries}\label{Sec:pre}

In this section eigenvectors of matrices are written as column vectors. Given a vector $\mathbf{x}$, denote by $\mathbf{x}^{T}$ and $\mathbf{x}^{H}$ the transpose and conjugate transpose of $\mathbf{x}$, respectively. Denote by $I_n$ the identity matrix of size $n$.

The spectral decomposition \cite{CGodsil} of symmetric matrices plays a key role in our subsequent discussion. So let us explain it first. Let $G$ be a graph on $n$ vertices. Let $\lambda_1,  \ldots , \lambda_s$ be distinct eigenvalues of $G$ with respective multiplicities $l_1,\ldots,l_s$, where $l_1+\cdots+l_s=n$. Take $\left\{\mathbf{x}_{1}^{(r)},\ldots,\mathbf{x}_{l_r}^{(r)}\right\}$ to be an orthonormal basis of the eigenspace associated with $\lambda_{r}$, $r=1,\ldots,s$, and set
$$
X_r = \pmat{\mathbf{x}_{1}^{(r)} & \cdots & \mathbf{x}_{l_r}^{(r)}}, ~r=1, \ldots , s.
$$
Set
$$
X = \pmat{
      X_1^H \\
      \vdots \\
      X_s^H},
$$
and
$$
Y = \pmat{X_1 & \cdots & X_{s}},
$$
where $X_i^{H}$ is the conjugate transpose of $X_i$ for each $i$.
Clearly, $X_r$ is an $n \times l_r$ matrix and both $X$ and $Y$ are $n \times n$ matrices.
Define
$$
E_r = \sum\limits_{i=1}^{l_r}\mathbf{x}_i^{(r)} (\mathbf{x}_i^{(r)})^H=X_rX_r^H, ~r=1, \ldots , s.
$$
Then $XY=I_n$, $E_r^{2} = E_r$, and $E_{r}E_{r'}=0$ for $r\neq r'$. Since $X$ is invertible, it follows that $YX=I_n$, or equivalently $\sum\limits_{r=1}^{s}E_r=I_n$. Using this and the definition of $E_r$ and $\mathbf{x}_{i}^{(r)}$, one can easily verify (see e.g. \cite{CGodsil}) that
$$
A_G = \sum\limits_{r=1}^{s}\lambda_{r}E_r.
$$
This expression is called the \emph{spectral decomposition} of $A_G$ with respect to distinct eigenvalues (see \cite[p.517]{MAALA}). It follows from this decomposition that
\begin{equation}
\label{H_ADec2}
H_{A_G}(t)=\sum_{k\geq 0}\dfrac{(-\mathrm{i})^{k}A_G^{k}t^{k}}{k!}=\sum_{k\geq 0}\dfrac{(-\mathrm{i})^{k}(\sum_{r=1}^{s}\lambda_{r}^{k}E_{r})t^{k}}{k!} =\sum_{r=1}^{s}\exp(-\mathrm{i}t\lambda_{r})E_r.
\end{equation}

\begin{example}
\label{Exampe:1}
{\em
Consider the complete graph $K_n$ of order $n \geq 3$. We have
\begin{equation}\label{H_ADec3}
E_{1}=\frac{1}{n}\left(
  \begin{array}{rrrrr}
    n-1 & -1 & -1 & \cdots & -1\\
    -1 & n-1 & -1 & \cdots & -1\\
    \vdots & \vdots & \vdots & \vdots & \vdots \\
    -1 & -1 & \cdots & -1 & n-1
  \end{array}
\right),\;
E_{2}=\frac{1}{n}\left(
  \begin{array}{rrrrr}
    1 & 1 & 1 & \cdots & 1\\
    1 & 1 & 1 & \cdots & 1\\
    \vdots & \vdots & \vdots & \vdots & \vdots \\
    1 & 1 & \cdots & 1 & 1
  \end{array}
\right).
\end{equation}
The spectral decomposition of $A_{K_n}$ is given by $A_{K_n}=-E_{1}+(n-1)E_{2}$. By \eqref{H_ADec2},
$$
H_{A_{K_n}}(t) =\exp(\mathrm{i}t)E_1+\exp(-(n-1)\mathrm{i}t)E_2.
$$
Clearly,  if $t=\frac{2 z \pi}{n}, z\in\mathbb{Z} $, then $H_{A_{K_n}}(t)$ has entries of unit modulus. In fact, $H_{A_{K_n}}\left(\frac{2z \pi}{n}\right)$ is a scalar multiple of the identity matrix. Putting $t=\frac{2(nz+g)\pi}{n}, ~g\in\mathbb{Z}, ~0< g\leq n-1,$ we have
$$
  H_{A_{K_n}}\left(\frac{2 (nz+g) \pi}{n}\right) =\exp\left(\frac{2g\pi}{n}\mathrm{i}\right)E_1+\exp\left(\frac{2g\pi}{n}\mathrm{i}\right)E_2
=\exp\left(\frac{2g\pi}{n}\mathrm{i}\right)I_n.
$$
In particular, when $g=0$, we obtain $H_{A_{K_n}}(2 z \pi)=I_n$. Therefore, $K_n$ does not exhibit PST but is periodic with period $t=\frac{2\pi}{n}$.
}
\qed\end{example}

Let $A=(a_{ij})_{m\times n}$ and $B$ be matrices. The \emph{Kronecker product} of $A$ and $B$ is defined as
\begin{align*}
  A\otimes B & =(a_{ij}B)= \left(
         \begin{array}{ccc}
          a_{11}B & \cdots & a_{1n}B \\
           \vdots &  & \vdots \\
           a_{m1}B & \cdots & a_{mn}B \\
         \end{array}
       \right).
\end{align*}
The following properties of the Kronecker product of matrices can be found in \cite[Section 4.2]{Horn91} and \cite{HF}:
\begin{itemize}
  \item[\rm (1)] $(A\otimes B)\otimes C=A\otimes( B\otimes C)$;
  \item[\rm (2)] $(A+B)\otimes C=A\otimes C+B\otimes C$, ~~$C\otimes(A+B)=C\otimes A+C\otimes B$;
  \item[\rm (3)] $(A\otimes B)^T=A^T\otimes B^T$;
  \item [\rm (4)]$(A\otimes B)(C\otimes D)=(AC)\otimes(BD)$;
   \item [\rm (5)] $(kA)\otimes B=A\otimes (kB)=k(A\otimes B)$.
\end{itemize}

\begin{lemma}
\label{3.1}
Let $A = A_G$ for a graph $G$, and let $I$ be the identity matrix of any size. Then
\begin{itemize}
  \item[\rm (a)] $H_{A\otimes I}(t)=H_{A}(t)\otimes I$;
  \item[\rm (b)] $H_{I\otimes A}(t)= I\otimes H_{A}(t)$;
  \item[\rm (c)] $H_{I\otimes A \otimes I}(t)= I\otimes H_{A}(t)\otimes I$.
\end{itemize}
\end{lemma}

\begin{proof}
By (2), (4) and (5) above, we have
\begin{eqnarray*}
  H_{A\otimes I}(t)&=&\sum_{k\geq 0}\dfrac{(-\mathrm{i})^{k}(A\otimes I)^{k}t^{k}}{k!}\\
 &=&\sum_{k\geq 0}\dfrac{(-\mathrm{i})^{k}\left(A^k\otimes I^k\right)t^{k}}{k!}\\
&=&\sum_{k\geq 0}\dfrac{\left((-\mathrm{i})^{k}A^{k}t^{k}\right)\otimes I}{k!}\\
&=&H_{A}(t)\otimes I.
\end{eqnarray*}
This proves (a). Similarly, one can prove (b) and (c).  \qed
\end{proof}

\begin{lemma}
\emph{(Cvetkovi\'{c} et al. \cite[Theorem 2.5.3]{SG1})}
\label{7}
Let $G_1,\ldots,G_d$ be graphs and let $\emptyset\neq\mathcal{A} \subseteq \mathbb{Z}_{2}^{d}\setminus \{\mathbf{0}\}$. Then the adjacency matrix of $\NEPS(G_1, \ldots, G_d; \mathcal{A})$ is given by
$$A =\sum_{\mathbf{a} \in \mathcal{A}} A_{G_1}^{a_1}\otimes\cdots\otimes A_{G_d}^{a_d}.$$
\end{lemma}

\begin{lemma}
\emph{(Cvetkovi\'{c} et al. \cite[Theorem 2.5.4]{SG1})}
\label{1}
Let $\emptyset\neq\mathcal{A} \subseteq \mathbb{Z}_{2}^{d}\setminus \{\mathbf{0}\}$. Let $G_i$ be a graph on $k_i \ge 2$ vertices, $i=1, \ldots, d$. Suppose that $\lambda_{i,1},  \ldots , \lambda_{i,s_i}$ are distinct eigenvalues of $A_{G_i}$ with respective multiplicities  $l_{i,1},\ldots,l_{i,s_i}$, where $l_{i,1}+\cdots+l_{i,s_i}=k_i$. Let $\left\{\mathbf{x}_{i,1}^{(t)},\ldots,\mathbf{x}_{i,l_{i,t}}^{(t)}\right\}$ be linearly independent eigenvectors of the eigenspace associated with $\lambda_{i,t}$, $t=1,\ldots,s_i$. Then the eigenvalues of $\NEPS(G_1, \ldots, G_d; \mathcal{A})$ consists of all possible values of $\Lambda_{j_1, \ldots,  j_d}$, where
$$
\Lambda_{j_1, \ldots , j_d}=\sum\limits_{\mathbf{a} \in \mathcal{A}}\lambda_{1,j_1}^{a_1}\cdots\lambda_{d,j_d}^{a_d},~j_h=1, \ldots , s_h,~h=1, \ldots, d,
$$
and the corresponding eigenvectors are, respectively,
$$
\mathbf{x}_{1,q_1}^{(j_1)}\otimes\cdots\otimes \mathbf{x}_{d,q_{l_d}}^{(j_d)},~q_h=1, \ldots , l_{h,j_h},~h=1, \ldots , d.
$$
\end{lemma}

Recall that the tensor product of graphs $G$ and $H$ is defined as $G \otimes H=\NEPS(G,H; \{(1,1)\})$. So by Lemma \ref{7} the adjacency matrix of $G \otimes H$ is $A_G\otimes A_H$.

\begin{lemma}\emph{(Godsil \cite[Lemma 16.1]{CGodsil}; Pal et al. \cite[Proposition 2.2]{HPal})}\label{alg}
Let $G$ and $H$ be graphs. Let
$$
A_G=\sum_{r}\lambda_{r}E_r
$$
and
$$
A_H=\sum_{s}\mu_{s}F_s
$$
be the spectral decompositions of $A_G$ and $A_H$, respectively. Then
$$
H_{A_{G\otimes H}}(t)=\sum_{r}E_r \otimes H_{A_H}(\lambda_{r}t) = \sum_{s}H_{A_G}(\mu_{s}t)\otimes F_s.
$$
\end{lemma}

\begin{lemma}\emph{(Pal and Bhattacharjya \cite[Proposition 4.1]{HPal1})}
\label{8}
Let $G_1,\ldots, G_d$ be graphs and let $\emptyset\neq\mathcal{A}\subseteq \mathbb{Z}_2^{d}\setminus \{\mathbf{0}\}$. Denote by $H_{\mathcal{A}}(t)$ the transition matrix of $\NEPS(G_1,\ldots,G_d;\mathcal{A})$ and by $H_{\mathbf{a}}(t)$ the transition matrix of $\NEPS(G_1,\ldots,G_d;\{\mathbf{a}\})$ for $\mathbf{a} \in \mathbb{Z}_2^d\setminus \{\mathbf{0}\}$. Then
$$
H_\mathcal{A}(t)=\prod_{\mathbf{a} \in \mathcal{A}}H_{\mathbf{a}}(t).
$$
Equivalently, for any partition $\{\mathcal{A}_1, \ldots, \mathcal{A}_m\}$ of $\mathcal{A}$, we have
$$
H_{\mathcal{A}}(t)=\prod_{i=1}^m H_{\mathcal{A}_{i}}(t).
$$
\end{lemma}


\section{Proofs of the main results}
\label{Sec:main}


\subsection{A few lemmas}
\label{subsec:lem}

We need the following six lemmas in the proofs of Theorems \ref{Theorem3}, \ref{Theorem4} and \ref{10-3}.

\begin{lemma}
\label{Lemma2.5}
Let $\mathbf{a}=(a_1,\ldots,a_r)\in \mathbb{Z}_2^r\setminus \{\mathbf{0}\}$.
Denote by $A_{\mathbf{a}}$ the adjacency matrix of $\NEPS(r\odot K_2;\{\mathbf{a}\})$.
Then
\begin{itemize}
  \item[\rm (a)] $\pm1$ are all distinct eigenvalues of $\NEPS(r\odot K_2;\{\mathbf{a}\})$;
  \item[\rm (b)] the spectral decomposition of $A_{\mathbf{a}}$ is given by
\begin{equation}
\label{Lem2.8-1}
A_{\mathbf{a}}=(-1)\cdot\left(\frac{1}{2}\left(I_{2^r} -\bigotimes_{j=1}^{r}\left(A_{K_2}\right)^{a_j}\right)\right)
+1\cdot\left(\frac{1}{2}\left(I_{2^r}+\bigotimes_{j=1}^{r}\left(A_{K_2}\right)^{a_j}\right)\right);
\end{equation}
\item[\rm (c)]
\begin{equation}
 H_{\mathbf{a}}\left(\pi\right) = -I_{2^r},
 \label{3-1Eq-1}
 \end{equation}
 \begin{equation}
 H_{\mathbf{a}}\left(\frac{\pi}{2}\right) = -\mathrm{i}\bigotimes_{j=1}^{r}(A_{K_2})^{a_j}.
 \label{3-1Eq-2}
 \end{equation}
\end{itemize}
\end{lemma}

\begin{proof}
(a) By Lemma \ref{7}, $A_{\mathbf{a}}=\bigotimes_{j=1}^{r}\left(A_{K_2}\right)^{a_j}$. Since $\pm1$ are eigenvalues of $A_{K_2}$ and $1$ is the eigenvalue of $I_2$, according to Lemma \ref{1}, we obtain (a) immediately.

(b) Note that
\begin{eqnarray}
 A_{\mathbf{a}} \cdot \left(\frac{1}{2}\left(I_{2^r}-\bigotimes_{j=1}^{r}\left(A_{K_2}\right)^{a_j}\right)\right)
 &=&\left(\bigotimes_{j=1}^{r}\left(A_{K_2}\right)^{a_j}\right)
  \left(\frac{1}{2}\left(I_{2^r}-\bigotimes_{j=1}^{r}\left(A_{K_2}\right)^{a_j}\right)\right)\nonumber\\
&=&\,\frac{1}{2}\left(\bigotimes_{j=1}^{r}\left(A_{K_2}\right)^{a_j}-\bigotimes_{j=1}^{r}\left(A_{K_2}\right)^{2a_j}\right)\nonumber\\
&=&\,(-1)\cdot\left(\frac{1}{2}\left(I_{2^r} -\bigotimes_{j=1}^{r}\left(A_{K_2}\right)^{a_j}\right)\right),\label{Lem2.8-2}
\end{eqnarray}
where in the last step we used the fact that $\left(A_{K_2}\right)^{2}=I_2$. Similarly,
 \begin{equation}
 \label{Lem2.8-3}
 A_{\mathbf{a}}\cdot\left(\frac{1}{2}\left(I_{2^r}+
  \bigotimes_{j=1}^{r}\left(A_{K_2}\right)^{a_j}\right)\right) = 1\cdot\left(\frac{1}{2}\left(I_{2^r} +\bigotimes_{j=1}^{r}\left(A_{K_2}\right)^{a_j}\right)\right).
\end{equation}
On the other hand, one can easily verify that
\begin{equation}
\label{Lem2.8-4}
\left(\frac{1}{2}\left(I_{2^r}-\bigotimes_{j=1}^{r}\left(A_{K_2}\right)^{a_j}\right)\right)^2 =\frac{1}{2}\left(I_{2^r}-\bigotimes_{j=1}^{r}\left(A_{K_2}\right)^{a_j}\right),
  \end{equation}
  \begin{equation}
  \label{Lem2.8-5}
  \left(\frac{1}{2}\left(I_{2^r}+
  \bigotimes_{j=1}^{r}\left(A_{K_2}\right)^{a_j}\right)\right)^2=\frac{1}{2}\left(I_{2^r}+
  \bigotimes_{j=1}^{r}\left(A_{K_2}\right)^{a_j}\right),
  \end{equation}
  \begin{equation}\label{Lem2.8-6}
  \left(\frac{1}{2}\left(I_{2^r}-
  \bigotimes_{j=1}^{r}\left(A_{K_2}\right)^{a_j}\right)\right)
  \left(\frac{1}{2}\left(I_{2^r}+
  \bigotimes_{j=1}^{r}\left(A_{K_2}\right)^{a_j}\right)\right)=0,
  \end{equation}
and
\begin{equation}
\label{Lem2.8-7}
\frac{1}{2}\left(I_{2^r}-\bigotimes_{j=1}^{r}\left(A_{K_2}\right)^{a_j}\right)
+\frac{1}{2}\left(I_{2^r}+\bigotimes_{j=1}^{r}\left(A_{K_2}\right)^{a_j}\right)=I_{2^r}.
\end{equation}
We then obtain (b) using (\ref{Lem2.8-2})-(\ref{Lem2.8-7}) and the definition of the spectral decomposition.

(c) By (\ref{H_ADec2}) and (\ref{Lem2.8-1}), we have
$$
H_{\mathbf{a}}(t) = \exp(\mathrm{i}t)\left(\frac{1}{2}\left(I_{2^r}-\bigotimes_{j=1}^{r}\left(A_{K_2}\right)^{a_j}\right)\right)
+\exp(-\mathrm{i}t)\left(\frac{1}{2}\left(I_{2^r}+\bigotimes_{j=1}^{r}\left(A_{K_2}\right)^{a_j}\right)\right).
$$
Evaluating at $t=\pi$ and $t = \frac{\pi}{2}$, we obtain \eqref{3-1Eq-1} and \eqref{3-1Eq-2}, respectively.
\qed\end{proof}

\begin{lemma}
\label{3}
Let $\mathbf{a}=(a_1,\ldots,a_d)\in \mathbb{Z}_2^d\setminus \{\mathbf{0}\}$. Denote by $H_{\mathbf{a}}(t)$ the transition matrix of $\NEPS(K_{n_1},\ldots,K_{n_d};\{\mathbf{a}\})$ and set $h=\gcd(n_1,\ldots,n_d)$. Then the following hold:
\begin{itemize}
  \item[\rm (a)] $\NEPS(K_{n_1},\ldots,K_{n_d};\{\mathbf{a}\})$ does not have PST \emph{(\cite[Corollary 2.14]{SZ})};
  \item[\rm (b)] $\NEPS(K_{n_1},\ldots,K_{n_d};\{\mathbf{a}\})$ is periodic with period $\frac{2t \pi}{h}$ for every non-zero integer $t$, and moreover
      $$
      H_{\mathbf{a}}\left(\frac{2t \pi}{h}\right) =\exp\left((-1)^{w(\mathbf{a})-1}\frac{2t \pi}{h}\mathrm{i}\right)I_{n_1\cdots n_d}.
      $$
\end{itemize}
\end{lemma}

\begin{proof}
We prove (b) by induction on $d$. If $d=1$, then $w(\mathbf{a})=1$ as $\mathbf{a} \neq \textbf{0}$. By Example  \ref{Exampe:1},
 $$H_{\mathbf{a}}\left(\frac{2t \pi}{n_1}\right) =\exp\left(\frac{2t \pi}{n_1}\mathrm{i}\right)I_{n_1} =\exp\left((-1)^{(1-1)}\frac{2t \pi}{n_1}\mathrm{i}\right)I_{n_1}.$$
So the result is true when $d=1$.

Assume that the result in (b) is true for some $d=l$. Consider $d=l+1$ and $\mathbf{a}=(a_1,\ldots,a_l,a_{l+1})$. Let ${\mathbf{a}}^{*}=(a_1,\ldots,a_l)$. Set $h_l=\gcd(n_1,\ldots, n_l)$.
\medskip

\emph{Case 1.} ${\mathbf{a}}^{*}\not=\mathbf{0}$.

In this case, by the hypothesis we have
$$
H_{\mathbf{a}^{*}}\left(\frac{2t \pi}{h_l}\right) = \exp\left((-1)^{w(\mathbf{a}^{\ast})-1}\frac{2t \pi}{h_l}\mathrm{i}\right)I_{n_1\cdots n_l}.
$$

\medskip
\emph{Case 1.1.} $a_{l+1}=1$.

In this case we have $w(\mathbf{a})=w(\mathbf{a}^{*})+1$. According to Lemma \ref{7}, $\NEPS(K_{n_1},\ldots,K_{n_{l+1}};\{\mathbf{a}\})$ is the tensor product $\NEPS(K_{n_1},\ldots,K_{n_{l}};\{\mathbf{a}^{*}\})\otimes K_{n_{l+1}}$. By Example  \ref{Exampe:1}, the spectral decomposition of $A_{K_{n_{l+1}}}$ is given by $A_{K_{n_{l+1}}}=-E_{1}+(n_{l+1}-1)E_{2}$, where $E_{1}$ and $E_{2}$ are as shown in (\ref{H_ADec3}) but with $n$ replaced by $n_{l+1}$. Note that $h_{l+1}=\gcd(n_1,\ldots, n_l,n_{l+1})=\gcd(h_l,n_{l+1})$. Set $h_l= k h_{l+1}$, where $k \ne 0$ is an integer. Lemma \ref{alg} implies that
  \begin{eqnarray*}
  H_{\mathbf{a}} \left(\frac{2t \pi}{h_{l+1}}\right) & = &\,H_{\mathbf{a}^{*}}\left(-\frac{2t \pi}{h_{l+1}}\right)\otimes E_1 +H_{\mathbf{a}^{*}}\left((n_{l+1}-1)\frac{2t \pi}{h_{l+1}}\right)\otimes E_2\\
  &=&\,H_{\mathbf{a}^{*}}\left(-\frac{2t \pi}{h_{l+1}}\right)\otimes E_1 +H_{\mathbf{a}^{*}}\left(-\frac{2t \pi}{h_{l+1}}\right)\otimes E_2\\
  &=&\,H_{\mathbf{a}^{*}}\left(-\frac{2t \pi}{h_{l+1}}\right)\otimes (E_1+E_2)\\
  &=&\, H_{\mathbf{a}^{*}}\left(-k\cdot\frac{2t \pi}{h_{l}}\right)\otimes I_{n_{l+1}}\\
  &=& \left(H_{\mathbf{a}^{*}}\left(-\frac{2t \pi}{h_{l}}\right)\right)^{k} \otimes I_{n_{l+1}}\\
  &=& \left(\exp\left((-1)^{w(\mathbf{a}^{\ast})-1+1}\frac{2t \pi}{h_l}\mathrm{i}\right)I_{n_1\cdots n_l}\right)^{k} \otimes I_{n_{l+1}}\\
&=&\exp\left((-1)^{(w(\mathbf{a})-1)}\frac{2t \pi}{h_{l}}\mathrm{i}\right)^{k} I_{n_1 \cdots n_{l+1}}\\
&=&\exp\left((-1)^{(w(\mathbf{a})-1)}\frac{2t \pi}{h_{l+1}}\mathrm{i}\right) I_{n_1 \cdots n_{l+1}}.
\end{eqnarray*}

\emph{Case 1.2.}  $a_{l+1}=0$.

In this case we have $w(\mathbf{a})=w(\mathbf{a}^{*})$. Let $A_{\mathbf{a}^{*}}$ be the adjacency matrix of  $\NEPS(K_{n_1},\ldots,K_{n_{l}};\{\mathbf{a}^{*}\})$. By Lemma \ref{7}, the adjacency matrix of  $\NEPS(K_{n_1},\ldots,K_{n_{l+1}};\{\mathbf{a}\})$ is $A_{\mathbf{a}^{*}}\otimes I_{n_{l+1}}$. According to Lemma \ref{3.1}(a) and the same technique as in Case 1.1, we obtain
 \begin{eqnarray*}
  H_\mathbf{a}\left(\frac{2t \pi}{h_{l+1}}\right)
  &=&\,H_{\mathbf{a}^{*}}\left(\frac{2t \pi}{h_{l+1}}\right)\otimes I_{n_{l+1}}\\
  &=&\exp\left((-1)^{(w(\mathbf{a}^{*})-1)}\frac{2t \pi}{h_{l+1}}\mathrm{i}\right) I_{n_1 \cdots n_{l+1}}\\
  &=&\exp\left((-1)^{(w(\mathbf{a})-1)}\frac{2t \pi}{h_{l+1}}\mathrm{i}\right) I_{n_1 \cdots n_{l+1}}.
\end{eqnarray*}

\emph{Case 2.} $\mathbf{a}^{*}=\mathbf{0}$.

In this case we have $\mathbf{a}=(0,\ldots,0,1)$ and so $w(\mathbf{a})=1$. By Lemma \ref{7}, the adjacency matrix of $\NEPS(K_{n_1},\ldots,K_{n_{l+1}};\{\mathbf{a}\})$ is $I_{n_1}\otimes\cdots\otimes I_{n_l}\otimes A_{K_{n_{l+1}}}$. Note that $h_{l+1}=\gcd(n_1,\ldots, n_l,n_{l+1})=\gcd(h_l,n_{l+1})$. Set $n_{l+1}= k h_{l+1}$, where $k \ne 0$ is an integer. According to Lemma \ref{3.1}(b) and Example  \ref{Exampe:1}, we have

 \begin{eqnarray*}
  H_\mathbf{a}\left(\frac{2t \pi}{h_{l+1}}\right) & = & I_{n_1}\otimes\cdots\otimes I_{n_l}\otimes H_{A_{K_{n_{l+1}}}}\left(\frac{2t \pi}{h_{l+1}}\right)\\
 &=& I_{n_1}\otimes\cdots\otimes I_{n_l}\otimes H_{A_{K_{n_{l+1}}}}\left(k \cdot\frac{2t \pi}{n_{l+1}}\right)\\
 &=& I_{n_1}\otimes\cdots\otimes I_{n_l}\otimes \left(H_{A_{K_{n_{l+1}}}}\left(\frac{2t \pi}{n_{l+1}}\right)\right)^{k}\\
 &=& I_{n_1}\otimes\cdots\otimes I_{n_l}\otimes \left(\exp\left(\frac{2t \pi}{n_{l+1}}\mathrm{i}\right)I_{n_{l+1}}\right)^{k}\\
 &=& I_{n_1}\otimes\cdots\otimes I_{n_l}\otimes \left(\exp\left(\frac{2t \pi}{h_{l+1}}\mathrm{i}\right)I_{n_{l+1}}\right)\\
 &=& \exp\left((-1)^{(w(\mathbf{a})-1)}\frac{2t \pi}{h_{l+1}}\mathrm{i}\right) I_{n_1 \cdots n_{l+1}}.
\end{eqnarray*}
This completes the proof. \qed
\end{proof}

\begin{lemma}\label{4}
Let $\mathbf{a}=(a_1,\ldots,a_d)\in \mathbb{Z}_2^d\setminus \{\mathbf{0}\}$ and $\mathbf{y}=(a_{d+1},\ldots,a_{d+r}) \in \mathbb{Z}_2^r\setminus \{\mathbf{0}\}$. Set  $h=\gcd(n_1,\ldots,n_d)$ and $\mathbf{b} = (a_1,\ldots,a_d,a_{d+1},\ldots,a_{d+r})$. Denote by $H_{\mathbf{b}}(t)$ the transition matrix of $G = \NEPS(K_{n_1},\ldots,K_{n_d},r\odot K_2;\{\mathbf{b}\})$. Then the following hold:
\begin{itemize}
  \item[\rm (a)] if $h$ is odd, then $G$ is periodic with period $2\pi$ and $H_\mathbf{b}(2\pi)=I_{2^rn_1\cdots n_d}$;
  \item[\rm (b)] if $h$ is even, then $G$ is periodic with period $\pi$ and $H_\mathbf{b}(\pi)=-I_{2^rn_1\cdots n_d}$;
\item[\rm (c)] if $h$ is a multiple of $4$, then $G$ exhibits PST at time $t=\frac{\pi}{2}$ and
      $$
 H_{\mathbf{b}}\left(\frac{\pi}{2}\right)
=(-1)^{(w(\mathbf{a})-1)}\mathrm{i} I_{n_1\cdots n_{d}}\otimes\left(\bigotimes_{j=d+1}^{d+r}(A_{K_2})^{a_j}\right).
$$
\end{itemize}

\end{lemma}
\begin{proof}
By Lemma \ref{7}, $G$ is isomorphic to the tensor product $\NEPS(K_{n_1}, \ldots,K_{n_d};\{\mathbf{a}\})\otimes \NEPS(r\odot K_2;\{\mathbf{y}\})$.
Let $H_{\mathbf{a}}(t)$ be the transition matrix of $\NEPS(K_{n_1},\ldots,K_{n_d};\{\mathbf{a}\})$. According to Lemmas \ref{alg} and \ref{Lemma2.5}, we have
\begin{equation}
\label{Eq:E1E2}
H_{\mathbf{b}}(t)=H_{\mathbf{a}}(-t)\otimes \left(\frac{1}{2}\left(I_{2^r}-\bigotimes_{j=d+1}^{d+r}\left(A_{K_2}\right)^{a_j}\right)\right)+ H_{\mathbf{a}}(t)\otimes \left(\frac{1}{2}\left(I_{2^r}+\bigotimes_{j=d+1}^{d+r}\left(A_{K_2}\right)^{a_j}\right)\right).
\end{equation}
By (\ref{Eq:E1E2}), if there exists an entry of $H_\mathbf{a}(t)$ having unit modulus, then there exists a pair of vertices $\mathbf{u}, \mathbf{v}$ of $G$ such that $\big|\frac{1}{2}(H_{\mathbf{a}}(-t))_{\mathbf{u}, \mathbf{v}}+\frac{1}{2}(H_{\mathbf{a}}(t))_{\mathbf{u}, \mathbf{v}}\big|=1$
or $\big|-\frac{1}{2}(H_{\mathbf{a}}(-t))_{\mathbf{u}, \mathbf{v}}+\frac{1}{2}(H_{\mathbf{a}}(t))_{\mathbf{u}, \mathbf{v}}\big|=1$. Note that $H_{\mathbf{a}}(-t)$ and $H_{\mathbf{a}}(t)$ are unitary matrices.
So $H_{\mathbf{a}}(t)$ has an entry with unit modulus if and only if $(H_{\mathbf{a}}(-t))_{\mathbf{u}, \mathbf{v}}=(H_{\mathbf{a}}(t))_{\mathbf{u}, \mathbf{v}}=\pm1$ or $(H_{\mathbf{a}}(-t))_{\mathbf{u}, \mathbf{v}}=-(H_{\mathbf{a}}(t))_{\mathbf{u}, \mathbf{v}}=\pm1$.
By Lemma \ref{3}(b), only if $\mathbf{u} = \mathbf{v}$ and $t=\frac{2k\pi}{h}~(k\in \mathbb{N}\setminus\{\mathbf{0}\})$, can the cases above happen.

\medskip
\emph{Case 1.}
$\left(H_{\mathbf{a}}(-\frac{2k\pi}{h})\right)_{\mathbf{u}, \mathbf{u}} =\left(H_{\mathbf{a}}(\frac{2k\pi}{h})\right)_{\mathbf{u}, \mathbf{u}}$.

According to Lemma \ref{3}, we have
$$
\exp\left((-1)^{w(\mathbf{a})}\frac{2k\pi}{h}\mathrm{i}\right) =\exp\left((-1)^{(w(\mathbf{a})-1)}\frac{2k\pi}{h}\mathrm{i}\right),
$$
which is equivalent to that
$\sin\left(\frac{2k\pi}{h}\right)=0$.
This equation is satisfied if and only if $\frac{2k\pi}{h}=l\pi,~l\in \mathbb{N}$.

If $h$ is odd, we choose $k=h$. By (\ref{Eq:E1E2}) and Lemma \ref{3}(b), we obtain
\begin{eqnarray*}
H_\mathbf{b}\left(2\pi\right) & = & H_{\mathbf{a}}\left(-2\pi\right)\otimes \left(\frac{1}{2}\left(I_{2^r}-\bigotimes_{j=d+1}^{d+r}\left(A_{K_2}\right)^{a_j}\right)\right)+ H_{\mathbf{a}}\left(2\pi\right)\otimes \left(\frac{1}{2}\left(I_{2^r}+\bigotimes_{j=d+1}^{d+r}\left(A_{K_2}\right)^{a_j}\right)\right)\\
 & = & I_{2^rn_1\cdots n_d},
\end{eqnarray*}
which means that $G$ is periodic with period $2\pi$, yielding (a).

Similarly, if $h$ is even, we choose $k=h/2$. Then
 \begin{eqnarray*}
  H_\mathbf{b}\left(\pi\right) & = & H_{\mathbf{a}}\left(-\pi\right)\otimes \left(\frac{1}{2}\left(I_{2^r}-\bigotimes_{j=d+1}^{d+r}\left(A_{K_2}\right)^{a_j}\right)\right)+ H_{\mathbf{a}}\left(\pi\right)\otimes \left(\frac{1}{2}\left(I_{2^r}+\bigotimes_{j=d+1}^{d+r}\left(A_{K_2}\right)^{a_j}\right)\right)\\
& = &-I_{2^rn_1\cdots n_d},
\end{eqnarray*}
which implies that $G$ is periodic with period $\pi$, yielding (b).

\medskip
\emph{Case 2.}  $\left(H_{\mathbf{a}}(-\frac{2k\pi}{h})\right)_{\mathbf{u}, \mathbf{u}} =-\left(H_{\mathbf{a}}(\frac{2k\pi}{h})\right)_{\mathbf{u}, \mathbf{u}}$.

According to Lemma \ref{3}, we have
$$\exp\left((-1)^{w(\mathbf{a})}\frac{2k\pi}{h}\mathrm{i}\right) =-\exp\left((-1)^{(w(\mathbf{a})-1)}\frac{2k\pi}{h}\mathrm{i}\right),$$
which is equivalent to that $\cos\left(\frac{2k\pi}{h}\right)=0$.
This equation holds if and only if $\frac{2k\pi}{h}=l\pi +\frac{\pi}{2},~l\in \mathbb{N}$, which happens only when $4$ divides $h$. Assume that $4$ divides $h$. Set $k=h/4$. According to (\ref{Eq:E1E2}) and  Lemma \ref{3}(b), we have
\begin{eqnarray*}
  H_\mathbf{b}\left(\frac{\pi}{2}\right) & = & H_{\mathbf{a}}\left(-\frac{\pi}{2}\right)\otimes \left(\frac{1}{2}\left(I_{2^r}-\bigotimes_{j=d+1}^{d+r}\left(A_{K_2}\right)^{a_j}\right)\right)+ H_{\mathbf{a}}\left(\frac{\pi}{2}\right)\otimes \left(\frac{1}{2}\left(I_{2^r}+\bigotimes_{j=d+1}^{d+r}\left(A_{K_2}\right)^{a_j}\right)\right)\\
& = & \mathrm{i}\sin\left((-1)^{(w(\mathbf{a})-1)}\frac{\pi}{2}\right)I_{n_1\cdots n_{d}}\otimes\left(\bigotimes_{j=d+1}^{d+r}(A_{K_2})^{a_j}\right)\\
& = & (-1)^{(w(\mathbf{a})-1)}\mathrm{i} I_{n_1\cdots n_{d}}\otimes\left(\bigotimes_{j=d+1}^{d+r}(A_{K_2})^{a_j}\right),
\end{eqnarray*}
which means that $G$ exhibits PST at time $t=\frac{\pi}{2}$, yielding (c). \qed
\end{proof}

\begin{lemma}\label{10-1}
Let $\mathbf{a}=(a_1,\ldots,a_d)\in \mathbb{Z}_2^d\setminus \{\mathbf{0}\}$ and $\mathbf{b} = (a_1,\ldots,a_d, 0,\ldots,0) \in \mathbb{Z}_2^{d+r}$. Then $\NEPS(K_{n_1},\ldots,K_{n_d},r\odot K_2;\{\mathbf{b}\})$ is periodic with period $\frac{2\pi}{h}$ and its transition matrix $H_{\mathbf{b}}(t)$ satisfies
$$
      H_\mathbf{b}\left(\frac{2\pi}{h}\right) =\exp\left((-1)^{w(\mathbf{a})-1}\frac{2\pi}{h}\mathrm{i}\right)I_{2^rn_1\cdots n_d}.
      $$
\end{lemma}

\begin{proof}
Denote by $A_{\mathbf{a}}$ and $A_{\mathbf{b}}$ the adjacency matrices of the graphs $\NEPS(K_{n_1},\ldots,K_{n_{d}};\{\mathbf{a}\})$ and $\NEPS(K_{n_1},\ldots,K_{n_{d}},r\odot K_2;\{\mathbf{b}\})$, respectively. Lemma \ref{7} implies that $A_{\mathbf{b}}=A_{\mathbf{a}}\otimes \underbrace{I_2\otimes\cdots\otimes I_2}_{r}$. According to Lemmas \ref{3.1}(a) and \ref{3}(b), we have
\begin{eqnarray*}
  H_\mathbf{b}\left(\frac{2\pi}{h}\right)
  &=& H_{\mathbf{a}}\left(\frac{2\pi}{h}\right)\otimes\underbrace{I_2\otimes\cdots\otimes I_2}_{r}\\
  &=& \exp\left((-1)^{w(\mathbf{a})-1}\frac{2\pi}{h}\mathrm{i}\right)I_{n_1\cdots n_d}\otimes\underbrace{I_2\otimes\cdots\otimes I_2}_{r}\\
  &=& \exp\left((-1)^{w(\mathbf{a})-1}\frac{2\pi}{h}\mathrm{i}\right)I_{2^rn_1\cdots n_d}.
\end{eqnarray*}
This implies that $\NEPS(K_{n_1},\ldots,K_{n_d},r\odot K_2;\{\mathbf{b}\})$ is periodic with period $\frac{2\pi}{h}$.
\qed
\end{proof}

\begin{lemma}
\label{10-2}
Let $\mathbf{y}=(a_{d+1},\ldots,a_{d+r}) \in \mathbb{Z}_2^r\setminus \{\mathbf{0}\}$ and $\mathbf{b} = (0,\ldots,0,a_{d+1},\ldots,a_{d+r}) \in \mathbb{Z}_2^{d+r}$.
Denote by $H_{\mathbf{b}}(t)$ the transition matrix of $\NEPS(K_{n_1},\ldots,K_{n_d},r\odot K_2;\{\mathbf{b}\})$. Then the following hold:
  \begin{itemize}
  \item[\rm (a)] $\NEPS(K_{n_1},\ldots,K_{n_d},r\odot K_2;\{\mathbf{b}\})$ exhibits PST at time $\frac{\pi}{2}$, and
 $$
 H_{\mathbf{b}}\left(\frac{\pi}{2}\right)
=-\mathrm{i} I_{ n_{1}\cdots n_{d}}\otimes\left(\bigotimes_{j=d+1}^{d+r}(A_{K_2})^{a_j}\right);
$$

\item[\rm (b)] $\NEPS(K_{n_1},\ldots,K_{n_d},r\odot K_2;\{\mathbf{b}\})$ is periodic with period $\pi$, and
$$
H_{\mathbf{b}}\left(\pi\right)
=-I_{2^rn_1\cdots n_d}.
$$
\end{itemize}
\end{lemma}

\begin{proof}
Let $A_{\mathbf{y}}$ and $A_{\mathbf{b}}$ denote respectively the adjacency matrices of $\NEPS(r\odot K_2;\{\mathbf{y}\})$ and $\NEPS(K_{n_{1}},\ldots, K_{n_{d}},r\odot K_2;\{\mathbf{b}\})$. Lemma \ref{7} implies that
 $A_{\mathbf{b}}=I_{n_1\cdots n_d}\otimes A_{\mathbf{y}}$. According to Lemma \ref{3.1}(b) and Equation  (\ref{3-1Eq-2}),  we have
\begin{eqnarray*}
  H_\mathbf{b}\left(\frac{\pi}{2}\right)
  &=& I_{n_1\cdots n_d}\otimes H_{\mathbf{y}}\left(\frac{\pi}{2}\right)\\
  &=& I_{n_1\cdots n_d}\otimes \left(-\mathrm{i}\bigotimes_{j=d+1}^{d+r}(A_{K_2})^{a_j}\right)\\
  &=& -\mathrm{i} I_{ n_{1}\cdots n_{d}}\otimes \left(\bigotimes_{j=d+1}^{d+r}(A_{K_2})^{a_j}\right),
\end{eqnarray*}
which implies that $\NEPS(K_{n_1},\ldots,K_{n_d},r\odot K_2;\{\mathbf{b}\})$ exhibits PST at time $\frac{\pi}{2}$, yielding (a).

By Lemma \ref{3.1}(b) and Equation (\ref{3-1Eq-1}), we have
\begin{eqnarray*}
  H_\mathbf{b}\left(\pi\right)
  & = & I_{n_1\cdots n_d}\otimes H_{\mathbf{y}}\left(\pi\right)=I_{n_1\cdots n_d}\otimes \left(-I_{2^r}\right) \\
  & = & -I_{2^rn_1\cdots n_d},
\end{eqnarray*}
which implies that $\NEPS(K_{n_1},\ldots,K_{n_d},r\odot K_2;\{\mathbf{b}\})$ is periodic with period $\pi$, yielding (b).
\qed
\end{proof}




\begin{lemma}\label{Idempotent}
Let $\emptyset \neq \mathcal{B} \subseteq \mathbb{Z}_2^{r}$ with $|\mathcal{B}|\ge2$. Take $\mathbf{a}=(a_1,\ldots,a_d)\in \mathbb{Z}_2^{d}\setminus \{\mathbf{0}\}$ and set
$$
\mathcal{B}(\mathbf{a}) = \{(a_1,\ldots,a_{d},a_{d+1},\ldots,a_{d+r}): (a_{d+1},\ldots,a_{d+r}) \in \mathcal{B}\}.
$$
Denote by $H_{\mathcal{B}(\mathbf{a})}(t)$ the transition matrix of $\NEPS(K_{n_1},\ldots,K_{n_d},r\odot K_2;\mathcal{B}(\mathbf{a}))$. Set
$h=\gcd (n_1,\ldots,n_d)$. Suppose that $\mathbf{c}(\mathcal{B}) =\mathbf{0}$. Then the following hold:
\begin{itemize}
\item[\rm (a)] if $h$ is odd and $|\mathcal{B}|\equiv 0~(\mod ~4)$, then $\NEPS(K_{n_1},\ldots,K_{n_d},r\odot K_2; \mathcal{B}(\mathbf{a}))$ is periodic with period $\frac{\pi}{2}$, and
$$
H_{\mathcal{B}(\mathbf{a})}\left(\frac{\pi}{2}\right)=I_{2^rn_1\cdots n_d};
$$
\item[\rm (b)] if $h$ is even but not a multiple of $4$, and $|\mathcal{B}|\equiv0~(\mod ~2)$, then $\NEPS(K_{n_1},\ldots,K_{n_d},r\odot K_2; \mathcal{B}(\mathbf{a}))$ is periodic with period $\frac{\pi}{2}$, and
$$
H_{\mathcal{B}(\mathbf{a})}\left(\frac{\pi}{2}\right)
=-I_{2^rn_1\cdots n_r}.
$$
\end{itemize}
\end{lemma}

\begin{proof}
Set $c = |\mathcal{B}(\mathbf{a})|$ and $\mathcal{B}(\mathbf{a})=\{\mathbf{b}^{(1)},\ldots,\mathbf{b}^{(c)}\}$ with $\mathbf{b}^{(i)}=(a_1,\ldots,a_{d},a_{d+1}^i,\ldots,a_{d+r}^i)$ for $i=1,2,\ldots, c$. By (\ref{Eq:E1E2}), we have
\begin{equation}
\label{Ef:E1E2}
H_{\mathbf{b}^{(i)}}(t)=(H_{\mathbf{a}}(t)-H_{\mathbf{a}}(-t)) \otimes \left(\frac{1}{2}\bigotimes_{j=d+1}^{d+r}\left(A_{K_2}\right)^{a_j^i}\right) +(H_{\mathbf{a}}(t)+H_{\mathbf{a}}(-t))\otimes \left(\frac{1}{2}I_{2^r}\right).
\end{equation}
Set $C = \{1,2,\ldots, c\}$.
According to Lemma \ref{8} and Equation (\ref{Ef:E1E2}), we have
\begin{align}
& H_{\mathcal{B}(\mathbf{a})} \left(\frac{\pi}{2}\right)
= \prod_{i=1}^{c}H_{\mathbf{b}^{(i)}}\left(\frac{\pi}{2}\right)\nonumber\\
= & \sum_{\emptyset \ne S \subsetneq C} \left(\left(H_{\mathbf{a}}\left(\frac{\pi}{2}\right)-H_{\mathbf{a}}\left(-\frac{\pi}{2}\right)\right)^{|S|} \otimes \left(\frac{1}{2^{|S|}}\bigotimes_{j=d+1}^{d+r}\left(A_{K_2}\right)^{\sum\limits_{i\in S}a_j^i}\right) \right)  \nonumber\\
 & \hspace{4cm} \times \left(\left(H_{\mathbf{a}}\left(\frac{\pi}{2}\right)+H_{\mathbf{a}}\left(-\frac{\pi}{2}\right)\right)^{c-|S|}\otimes \frac{1}{2^{c-|S|}}I_{2^r}\right)  \nonumber \\
  &+\left(H_{\mathbf{a}}\left(\frac{\pi}{2}\right)+H_{\mathbf{a}}\left(-\frac{\pi}{2}\right)\right)^{c}\otimes \frac{1}{2^{c}}I_{2^r}  + \left(H_{\mathbf{a}}\left(\frac{\pi}{2}\right)-H_{\mathbf{a}}\left(-\frac{\pi}{2}\right)\right)^{c} \otimes \left(\frac{1}{2^{c}}\bigotimes_{j=d+1}^{d+r} \left(A_{K_2}\right)^{\sum\limits_{i\in C}a_j^i}\right) \nonumber \\
  = & \sum_{\emptyset \ne S \subsetneq C} \left(\left(H_{\mathbf{a}}\left(\frac{\pi}{2}\right) -H_{\mathbf{a}}\left(-\frac{\pi}{2}\right)\right)^{|S|}\left(H_{\mathbf{a}}\left(\frac{\pi}{2}\right) +H_{\mathbf{a}}\left(-\frac{\pi}{2}\right)\right)^{c-|S|}\right) \otimes \left(\frac{1}{2^{c}}\bigotimes_{j=d+1}^{d+r} \left(A_{K_2}\right)^{\sum\limits_{i\in S}a_j^i}\right)  \nonumber \\
  &+\left(H_{\mathbf{a}}\left(\frac{\pi}{2}\right) +H_{\mathbf{a}}\left(-\frac{\pi}{2}\right)\right)^{c}\otimes \frac{1}{2^{c}}I_{2^r}  + \left(H_{\mathbf{a}}\left(\frac{\pi}{2}\right) -H_{\mathbf{a}}\left(-\frac{\pi}{2}\right)\right)^{c} \otimes \left(\frac{1}{2^{c}} \bigotimes_{j=d+1}^{d+r}\left(A_{K_2}\right)^{\sum\limits_{i\in C}a_j^i}\right). \nonumber \\
\label{Idempotent2-1}
\end{align}
By Lemma \ref{3}(b), we have $H_{\mathbf{a}}(2\pi)=I_{n_1\ldots n_d}$. Hence
\begin{align}
 & \left(H_{\mathbf{a}}\left(\frac{\pi}{2}\right)-H_{\mathbf{a}}\left(-\frac{\pi}{2}\right)\right)
  \left(H_{\mathbf{a}}\left(\frac{\pi}{2}\right)+H_{\mathbf{a}}\left(-\frac{\pi}{2}\right)\right) \nonumber\\
  =&~ H_{\mathbf{a}}\left(\frac{\pi}{2}\right)H_{\mathbf{a}}\left(\frac{\pi}{2}\right) - H_{\mathbf{a}}\left(-\frac{\pi}{2}\right)H_{\mathbf{a}}\left(-\frac{\pi}{2}\right)\nonumber\\
  =&~ H_{\mathbf{a}}(\pi)-H_{\mathbf{a}}(-\pi)\nonumber\\
  =&~ H_{\mathbf{a}}(2\pi)H_{\mathbf{a}}(-\pi)-H_{\mathbf{a}}(-\pi)\nonumber\\
  =&~ 0.\label{Idempotent3-1}
\end{align}
Recall that  $\mathbf{c}(\mathcal{B}) =\mathbf{0}$.
Equations (\ref{Idempotent2-1}) and (\ref{Idempotent3-1}) imply that
  \begin{align}
  H_{\mathcal{B}(\mathbf{a})} \left(\frac{\pi}{2}\right)
  =&\left(H_{\mathbf{a}}\left(\frac{\pi}{2}\right) +H_{\mathbf{a}}\left(-\frac{\pi}{2}\right)\right)^{c}\otimes \frac{1}{2^{c}}I_{2^r}  + \left(H_{\mathbf{a}}\left(\frac{\pi}{2}\right) -H_{\mathbf{a}}\left(-\frac{\pi}{2}\right)\right)^{c} \otimes \left(\frac{1}{2^{c}} \bigotimes_{j=d+1}^{d+r}\left(A_{K_2}\right)^{\sum\limits_{i\in C}a_j^i}\right)
\nonumber\\
 =&\left(\left(H_{\mathbf{a}}\left(\frac{\pi}{2}\right)+
H_{\mathbf{a}}\left(-\frac{\pi}{2}\right)\right)^{c}+
\left(H_{\mathbf{a}}\left(\frac{\pi}{2}\right)-
H_{\mathbf{a}}\left(-\frac{\pi}{2}\right)\right)^{c}\right)\otimes \frac{1}{2^{c}}I_{2^r}\nonumber\\
 =&\left(\sum_{j=0}^{c} \binom{c}{j} H_{\mathbf{a}}\left(\frac{\pi}{2}\right)^{j}H_{\mathbf{a}}\left(-\frac{\pi}{2}\right)^{c-j}
 + \sum_{i=0}^{c} \binom{c}{j} H_{\mathbf{a}}\left(\frac{\pi}{2}\right)^{j}\left(-H_{\mathbf{a}}\left(-\frac{\pi}{2}\right)\right)^{c-j}\right) \otimes\frac{1}{2^{c}}I_{2^r}\nonumber\\
  =& \left(2\sum_{j=0 \atop j\equiv0 \,(\mod\,2)}^{c} \binom{c}{j} H_{\mathbf{a}}\left(\frac{j\pi}{2}\right)H_{\mathbf{a}}\left(-\frac{(c-j)\pi}{2}\right)\right) \otimes\frac{1}{2^{c}}I_{2^r}
\nonumber\\
 =&\left(2\sum_{j=0 \atop j\equiv0\,(\mod\,2)}^{c} \binom{c}{j}H_{\mathbf{a}}\left(-\frac{(c-2j)\pi}{2}\right)\right) \otimes\frac{1}{2^{c}}I_{2^r}.\label{Idempotent4-1}
\end{align}

If $h$ is odd and $|\mathcal{B}|\equiv 0~(\mod ~4)$, then by Equation (\ref{Idempotent4-1}) and Lemma \ref{3}(b) we have $H_{\mathcal{B}(\mathbf{a})}\left(\frac{\pi}{2}\right)=I_{2^rn_1\cdots n_r}$, yielding (a).
If $h$ is even but not a multiple of $4$, and $|\mathcal{B}|\equiv 0~(\mod ~2)$, then by Equation (\ref{Idempotent4-1}) and Lemma \ref{3}(b) we have $H_{\mathcal{B}(\mathbf{a})}\left(\frac{\pi}{2}\right)= -I_{2^rn_1\cdots n_r}$, yielding (b).
\qed
\end{proof}


\subsection{Proofs of Theorems \ref{Theorem3}, \ref{Theorem4} and \ref{10-3}}
\label{subsec:pf}

Now we are ready to prove Theorems \ref{Theorem3},  \ref{Theorem4} and \ref{10-3}.

\medskip

\begin{Tproof}\textbf{ of Theorem \ref{Theorem3}.}
Denote by $H_{\mathcal{A}}(t)$ and $H_{\mathbf{a}}(t)$ respectively the transition matrices of $\NEPS(K_{n_{1}},\ldots,K_{n_{d}},r\odot K_2; \mathcal{A})$ and $\NEPS(K_{n_{1}},\ldots,K_{n_{d}},r\odot K_2;\{\mathbf{a}\})$. Since the last $r$ coordinates of each element of $\mathcal{A}$ are equal to $0$, according to Lemmas \ref{8} and \ref{10-1}, we have
 \begin{eqnarray}
 H_{\mathcal{A}}\left(\frac{2\pi}{h}\right) & = & \prod_{\mathbf{a} \in \mathcal{A}}\left(H_{\mathbf{a}}\left(\frac{2\pi}{h}\right)\right)\nonumber\\
& = & \prod_{\mathbf{a} \in \mathcal{A}}\exp\left((-1)^{w(\mathbf{a})-1}\frac{2\pi}{h}\mathrm{i}\right)I_{2^rn_1\cdots n_d}\nonumber\\
& = & \exp\left(\sum_{\mathbf{a} \in \mathcal{A}}\left((-1)^{(w(\mathbf{a})-1)}\frac{2\pi}{h}\mathrm{i}\right)\right)I_{2^rn_1\cdots n_d}.\label{10-3-2}
\end{eqnarray}
This implies that $\NEPS(K_{n_1},\ldots,K_{n_d},r\odot K_2;\mathcal{A})$ is periodic with period $\frac{2\pi}{h}$.
\qed
\end{Tproof}

\begin{Tproof}\textbf{ of Theorem \ref{Theorem4}.}
Denote by $H_{\mathcal{A}}(t)$ and $H_{\mathbf{a}}(t)$ respectively the transition matrices of $\NEPS(K_{n_{1}},\ldots,K_{n_{d}},r\odot K_2; \mathcal{A})$ and $\NEPS(K_{n_{1}},\ldots,K_{n_{d}},r\odot K_2;\{\mathbf{a}\})$. Since the first $d$ coordinates of each element of $\mathcal{A}$ are equal to $0$, according to Lemmas \ref{8} and \ref{10-2}(a), we have
 \begin{eqnarray}
 H_{\mathcal{A}}\left(\frac{\pi}{2}\right) &=&\prod_{\mathbf{a} \in \mathcal{A}}\left(H_{\mathbf{a}}\left(\frac{\pi}{2}\right)\right)\nonumber\\
&=&\prod_{\mathbf{a} \in \mathcal{A}}\left(-\mathrm{i} I_{ n_{1}\cdots n_{d}}\otimes\left(\bigotimes_{j=d+1}^{d+r}(A_{K_2})^{a_j}\right)\right)\nonumber\\
&=& (-\mathrm{i})^{\mid\mathcal{A}\mid}
I_{n_1\cdots n_d}\otimes\left(\bigotimes_{j=d+1}^{d+r}(A_{K_2})^{\sum_{\mathbf{a} \in \mathcal{A}}a_j}\right).\label{10-3-3}
\end{eqnarray}
Thus, if $\mathbf{c}(\mathcal{A}) \neq \mathbf{0}$, then $\NEPS(K_{n_1},\ldots,K_{n_d},r\odot K_2;\mathcal{A})$ has PST from vertex $\mathbf{u}$ to vertex $\mathbf{u}+\mathbf{c}(\mathcal{A})$ at time $\frac{\pi}{2}$, yielding (a). On the other hand, if $\mathbf{c}(\mathcal{A}) = \mathbf{0}$, then $\NEPS(K_{n_1},\ldots,K_{n_d},r\odot K_2;\mathcal{A})$ is periodic with period $\frac{\pi}{2}$, yielding (b).
\qed
\end{Tproof}

\begin{Tproof}\textbf{ of Theorem \ref{10-3}.}
Set
\begin{eqnarray*}
\mathcal{A}_1 &=& \left\{(a_{1}, \ldots, a_d, a_{d+1}, \ldots, a_{d+r}) \in \mathcal{A}: (a_{d+1},\ldots, a_{d+r})=\mathbf{0}\right\},\\
\mathcal{A}_2 &=& \left\{(a_{1}, \ldots, a_d, a_{d+1}, \ldots, a_{d+r}) \in \mathcal{A}: (a_{1},\ldots, a_{d})=\mathbf{0}\right\},\\
\mathcal{A}_3 &= & \mathcal{A}\setminus(\mathcal{A}_1\cup\mathcal{A}_2).
 \end{eqnarray*}
Since $\mathcal{A}$ contains at least one element whose last $r$ coordinates are not all $0$ and at least one element whose first $d$ coordinates are not all $0$, we have $\mathcal{A}_1 \neq \mathcal{A}$ and $\mathcal{A}_2 \neq \mathcal{A}$. Since $\mathcal{A} \neq \emptyset$, at least one of $\mathcal{A}_1$, $\mathcal{A}_2$ and $\mathcal{A}_3$ is nonempty.
Denote by $H_{\mathcal{A}}(t)$, $H_{\mathcal{A}_i}(t)$ and $H_{\mathbf{a}}(t)$ the transition matrices of $\NEPS(K_{n_{1}},\ldots,K_{n_{d}},r\odot K_2; \mathcal{A})$, $\NEPS(K_{n_{1}},\ldots,K_{n_{d}},r\odot K_2; \mathcal{A}_i)$ and $\NEPS(K_{n_{1}},\ldots,K_{n_{d}},r\odot K_2;\{\mathbf{a}\})$, respectively.
By Lemma \ref{8}, we have
\begin{equation}\label{10-3-1}
 H_{\mathcal{A}}(t)=\prod_{i=1}^{3}H_{\mathcal{A}_{i}}(t),
\end{equation}
where $H_{\emptyset}(t)$ is understood as $I$.

Recall that
$$
\mathcal{A}^{*} =\{(a_1,\ldots,a_d): (a_1,\ldots,a_d,a_{d+1},\ldots,a_{d+r})\in \mathcal{A}\},
$$
$$
\mathcal{A}_{-}{(\mathbf{x})} = \{(x_{d+1},\ldots,x_{d+r}): (x_1,\ldots,x_d,x_{d+1},\ldots,x_{d+r})\in \mathcal{A} \text{ and } (x_1,\ldots,x_d) = \mathbf{x}\}
$$
and
$$
\mathcal{A}_{-}{(\mathbf{0})} = \{(x_{d+1},\ldots,x_{d+r}): (0,\ldots,0,x_{d+1},\ldots,x_{d+r})\in \mathcal{A}\}.
$$
Set
$$
\mathcal{A}_{+}{(\mathbf{x})} = \left\{(x_{1},\ldots, x_{d}, x_{d+1}, \ldots, x_{d+r}) \in \mathcal{A}: (x_{1},\ldots, x_{d})=\mathbf{x}\right\}
$$
and
$$
\mathcal{A}_{+}{(\mathbf{0})} = \left\{(0,\ldots,0,x_{d+1} \ldots, x_{d+r}): (0,\ldots,0,x_{d+1} \ldots, x_{d+r}) \in \mathcal{A}\right\}.
$$
Denote by $H_{\mathcal{A}_{+}(\mathbf{x})}(t)$ and $H_{\mathcal{A}_{+}(\mathbf{0})}(t)$ the transition matrices of $\NEPS(K_{n_{1}},\ldots,K_{n_{d}},r\odot K_2; \mathcal{A}_{+}(\mathbf{x}))$ and $\NEPS(K_{n_{1}},\ldots,K_{n_{d}},r\odot K_2;\mathcal{A}_{+}(\mathbf{0}))$, respectively.
By Lemma \ref{8}, we have
\begin{equation}\label{odd}
H_{\mathcal{A}}(t)=\left(\prod_{0\not=\mathbf{x} \in \mathcal{A}^*} H_{\mathcal{A}_{+}(\mathbf{x})}(t) \right) H_{\mathcal{A}_{+}(\mathbf{0})}(t).
\end{equation}
Recall also that $\mathbf{c}(\mathcal{A}) = (b_1, \ldots, b_d, b_{d+1}, \ldots,b_{d+r})$ and $\mathbf{c}(\mathcal{A}_{-}{(\mathbf{0})}) = (c_{d+1},\ldots,c_{d+r})$.

(a) If $h$ is odd, then according to Equation (\ref{10-3-1}) and Lemmas \ref{8}, \ref{4}(a), \ref{10-1} and \ref{10-2}(b), we have
\begin{eqnarray*}
H_{\mathcal{A}}\left(2\pi\right) &=& H_{\mathcal{A}_{1}}(2\pi)H_{\mathcal{A}_{2}}(2\pi)H_{\mathcal{A}_{3}}(2\pi)\\
&=& \left(\prod_{\mathbf{a} \in\mathcal{A}_{1} }H_{\mathbf{a}}\left(2\pi\right)\right)\left(\prod_{\mathbf{a} \in\mathcal{A}_{2} }H_{\mathbf{a}}\left(2\pi\right)\right)\left(\prod_{\mathbf{a}\in\mathcal{A}_{3} }H_{\mathbf{a}}\left(2\pi\right)\right)\\
&=& \left(\prod_{\mathbf{a} \in\mathcal{A}_{1} }H_{\mathbf{a}}\left(2\pi\right)\right)\left(\prod_{\mathbf{a} \in\mathcal{A}_{2} }\left(H_{\mathbf{a}}\left(\pi\right)\right)^2\right)\left(\prod_{\mathbf{a} \in\mathcal{A}_{3} }H_{\mathbf{a}}\left(2\pi\right)\right)\\
&=& I_{2^rn_1\cdots n_d}^{\mid \mathcal{A}_{1} \mid}\cdot \left(-I_{2^rn_1\cdots n_d}\right)^{2\mid \mathcal{A}_{2} \mid}\cdot I_{2^rn_1\cdots n_d}^{\mid \mathcal{A}_{3} \mid}\\
&=& I_{2^rn_1\cdots n_d}.
\end{eqnarray*}
This implies that $\NEPS(K_{n_{1}},\ldots,K_{n_{d}},r\odot K_2;\mathcal{A})$ is periodic with period  $2\pi$.

Since we assume that $|\mathcal{A}_{-}(\mathbf{x})| \equiv 0 ~(\mod ~4)$ and $\mathbf{c}(\mathcal{A}_{-}(\mathbf{x})) = \mathbf{0}$ for every $\mathbf{x} \in \mathcal{A}^* \setminus \{\mathbf{0}\}$, according to Lemmas \ref{8} and \ref{Idempotent}(a), Theorem \ref{Theorem4}, Equation (\ref{odd}) and the definition of $\mathcal{A}_{+}{(\mathbf{0})}$, we obtain
\begin{eqnarray}
H_{\mathcal{A}}\left(\frac{\pi}{2}\right) &=&\left(\prod_{0\not= \mathbf{x} \in \mathcal{A}^* }I_{2^rn_1\ldots n_r}\right)\cdot \left((-\mathrm{i})^{|\mathcal{A}_{+}{(\mathbf{0})}|}
I_{n_1\cdots n_d}\otimes\left(\bigotimes_{j=d+1}^{d+r}(A_{K_2})^{\sum_{\mathbf{y} \in \mathcal{A}_{+}{(\mathbf{0})}}y_j}\right)\right)\nonumber\\
&=&I_{2^rn_1\ldots n_r}\cdot \left((-\mathrm{i})^{|\mathcal{A}_{+}{(\mathbf{0})}|}
I_{n_1\cdots n_d}\otimes\left(\bigotimes_{j=1}^{r}(A_{K_2})^{\sum_{\mathbf{y} \in \mathcal{A}_{-}{(\mathbf{0})}}y_j}\right)\right) \nonumber\\
&=& (-\mathrm{i})^{|\mathcal{A}_{+}{(\mathbf{0})}|}
I_{n_1\cdots n_d}\otimes\left(\bigotimes_{j=1}^{r}(A_{K_2})^{\sum_{\mathbf{y} \in \mathcal{A}_{-}{(\mathbf{0})}}y_j}\right).\label{10-3-8}
\end{eqnarray}
Thus, if $(c_{d+1}, \ldots, c_{d+r}) \neq \mathbf{0}$, then, by (\ref{10-3-8}), $\NEPS(K_{n_1},\ldots,K_{n_d},r\odot K_2;\mathcal{A})$ has PST from vertex $\mathbf{u}$ to vertex $\mathbf{u}+(0, \ldots, 0, c_{d+1}, \ldots, c_{d+r})$ at time $\frac{\pi}{2}$, yielding (a.i). On the other hand, if $(c_{d+1}, \ldots, c_{d+r}) = \mathbf{0}$, then $\NEPS(K_{n_1},\ldots,K_{n_d},r\odot K_2;\mathcal{A})$ is periodic with period $\frac{\pi}{2}$, yielding (a.ii).

(b) If $h$ is even, then according to Equation (\ref{10-3-1}) and Lemmas \ref{8}, \ref{4}(b), \ref{10-1} and \ref{10-2}(b), we have
\begin{eqnarray*}
H_{\mathcal{A}}\left(\pi\right) &=& H_{\mathcal{A}_{1}}(\pi)H_{\mathcal{A}_{2}}(\pi)H_{\mathcal{A}_{3}}(\pi)\\
&=& \left(\prod_{\mathbf{a}\in\mathcal{A}_{1} }H_{\mathbf{a}}\left(\pi\right)\right)\left(\prod_{\mathbf{a} \in\mathcal{A}_{2} }H_{\mathbf{a}}\left(\pi\right)\right)\left(\prod_{\mathbf{a} \in\mathcal{A}_{3} }H_{\mathbf{a}}\left(\pi\right)\right)\\
&=& \left(-I_{2^rn_1\cdots n_d}\right)^{\mid \mathcal{A}_{1} \mid}\cdot \left(-I_{2^rn_1\cdots n_d}\right)^{\mid \mathcal{A}_{2} \mid}\cdot \left(-I_{2^rn_1\cdots n_d}\right)^{\mid \mathcal{A}_{3} \mid}\\
&=& (-1)^{\mid\mathcal{A}\mid}I_{2^rn_1\cdots n_d}.
\end{eqnarray*}
This implies that $\NEPS(K_{n_{1}},\ldots,K_{n_{d}},r\odot K_2;\mathcal{A})$ is periodic with  period $\pi$, yielding (b).

Suppose that $h$ is not a multiple of $4$. Since we assume that $|\mathcal{A}_{-}(\mathbf{x})| \equiv 0 ~(\mod~2)$ and $\mathbf{c}(\mathcal{A}_{-}(\mathbf{x})) = \mathbf{0}$ for every $\mathbf{x} \in \mathcal{A}^* \setminus \{\mathbf{0}\}$, by Lemmas \ref{8} and \ref{Idempotent}(b), Theorem \ref{Theorem4}, Equation (\ref{odd}) and the definition of $\mathcal{A}_+(\mathbf{0})$, we obtain
\begin{eqnarray}
H_{\mathcal{A}}\left(\frac{\pi}{2}\right)
&=&\prod_{0\not=\mathbf{x} \in \mathcal{A}^*}\left(-I_{2^rn_1\ldots n_r}\right)\cdot \left((-\mathrm{i})^{|\mathcal{A}_{+}{(\mathbf{0})}|}
I_{n_1\cdots n_d}\otimes\left(\bigotimes_{j=d+1}^{d+r}(A_{K_2})^{\sum_{\mathbf{y} \in \mathcal{A}_{+}{(\mathbf{0})}} y_j}\right)\right) \nonumber\\
&=& (-1)^{|\mathcal{A}^*|-1}I_{2^rn_1\ldots n_r}\cdot \left((-\mathrm{i})^{|\mathcal{A}_{+}{(\mathbf{0})}|}
I_{n_1\cdots n_d}\otimes\left(\bigotimes_{j=1}^{r}(A_{K_2})^{\sum_{\mathbf{y} \in \mathcal{A}_{-}{(\mathbf{0})}}y_j}\right)\right)\nonumber\\
&=& (-1)^{|\mathcal{A}^*|-1+|\mathcal{A}_{+}{(\mathbf{0})}|}\mathrm{i}^{|\mathcal{A}_{+}{(\mathbf{0})}|}
I_{n_1\cdots n_d}\otimes\left(\bigotimes_{j=1}^{r}(A_{K_2})^{\sum_{\mathbf{y} \in \mathcal{A}_{-}{(\mathbf{0})}}y_j}\right).\label{10-3-9}
\end{eqnarray}
Thus, if $(c_{d+1}, \ldots, c_{d+r}) \neq \mathbf{0}$, then according to (\ref{10-3-9}), $\NEPS(K_{n_1},\ldots,K_{n_d},r\odot K_2;\mathcal{A})$ has PST from vertex $\mathbf{u}$ to vertex $\mathbf{u}+(0, \ldots, 0, c_{d+1}, \ldots, c_{d+r})$ at time $\frac{\pi}{2}$, yielding (b.i). On the other hand, if $(c_{d+1}, \ldots, c_{d+r}) = \mathbf{0}$, then $\NEPS(K_{n_1},\ldots,K_{n_d},r\odot K_2;\mathcal{A})$ is periodic with period $\frac{\pi}{2}$, yielding (b.ii).

(c) Assume that $h$ is a multiple of $4$. If $\mathcal{A}_{3}\not=\emptyset$, then Lemmas \ref{8} and \ref{4}(c) imply that
 \begin{eqnarray}
 H_{\mathcal{A}_3}\left(\frac{\pi}{2}\right) &=& \prod_{\mathbf{a}\in \mathcal{A}_3}\left(H_{\mathbf{a}}\left(\frac{\pi}{2}\right)\right)\nonumber\\
&=& \prod_{\mathbf{a}\in \mathcal{A}_3}\left((-1)^{(w(a_{1},\ldots,a_{d})-1)}\mathrm{i} I_{n_{1}\cdots n_{d}}\otimes\left(\bigotimes_{j=d+1}^{d+r}(A_{K_2})^{a_j}\right)\right)\nonumber\\
&=& \left(-1\right)^{\sum_{\mathbf{a} \in \mathcal{A}_3}(w(a_{1},\ldots,a_{d})-1)}
\mathrm{i}^{\mid\mathcal{A}_3\mid}
I_{n_1\cdots n_d}\otimes\left(\bigotimes_{j=d+1}^{d+r}(A_{K_2})^{\sum_{\mathbf{a} \in \mathcal{A}_3}a_j} \right).\label{10-3-4}
\end{eqnarray}
If $\mathcal{A}_{1}\not=\emptyset$, then by (\ref{10-3-2}), we have
\begin{equation}\label{The10-3-5}
H_{\mathcal{A}_1}\left(\frac{\pi}{2}\right)= (-1)^{\sum_{\mathbf{a} \in \mathcal{A}_1}(w(\mathbf{a})-1)}\mathrm{i}^{\mid\mathcal{A}_1\mid}I_{2^rn_1\cdots n_d}.
\end{equation}
If $\mathcal{A}_{2}\not=\emptyset$, then by (\ref{10-3-3}), we have
\begin{equation}
\label{The10-3-6}
H_{\mathcal{A}_2}\left(\frac{\pi}{2}\right)= (-1)^{\mid\mathcal{A}_2\mid}\mathrm{i}^{\mid\mathcal{A}_2\mid}
I_{n_1\cdots n_d}\otimes\left(\bigotimes_{j=d+1}^{d+r}(A_{K_2})^{\sum_{\mathbf{a} \in \mathcal{A}_2}a_j}\right).
\end{equation}
Since at least one of $\mathcal{A}_1$, $\mathcal{A}_2$ and $\mathcal{A}_3$ is nonempty, according to Equations (\ref{10-3-1})-(\ref{The10-3-6}), we obtain
\begin{eqnarray}
H_\mathcal{A}\left(\frac{\pi}{2}\right) & = &\left((-1)^{\sum_{\mathbf{b}\in \mathcal{A}_1}(w(\mathbf{b})-1)}\mathrm{i}^{\mid\mathcal{A}_1\mid}I_{2^rn_1\cdots n_d}\right)\nonumber\\
  & & \cdot \left((-1)^{\mid\mathcal{A}_2\mid}\mathrm{i}^{\mid\mathcal{A}_2\mid}
I_{n_1\cdots n_d}\otimes\left(\bigotimes_{j=d+1}^{d+r}(A_{K_2})^{\sum_{\mathbf{b} \in \mathcal{A}_2}b_j}\right)\right) \nonumber\\
 & & \cdot \left((-1)^{\sum_{\mathbf{b}\in \mathcal{A}_3}(w(b_{1},\ldots,b_{d})-1)}
\mathrm{i}^{\mid\mathcal{A}_3\mid}
I_{n_1\cdots n_d}\otimes\left(\bigotimes_{j=d+1}^{d+r}(A_{K_2})^{\sum_{\mathbf{b} \in \mathcal{A}_3}b_j}\right)\right)\nonumber\\
& = & (-1)^{\sum_{\mathbf{b}\in \mathcal{A}_1}(w(\mathbf{b})-1)+\mid\mathcal{A}_2\mid+\sum_{\mathbf{b} \in \mathcal{A}_3}(w( b_{1},\ldots,b_{d})-1)}
\mathrm{i}^{\mid\mathcal{A}_1\mid+\mid\mathcal{A}_2\mid+\mid\mathcal{A}_3\mid}\nonumber\\
& & \cdot I_{n_1\cdots n_d}\otimes\left(\bigotimes_{j=d+1}^{d+r}(A_{K_2})^{\sum_{\mathbf{b} \in \mathcal{A}}b_j}\right)\nonumber\\
&=& (-1)^{\sum_{\mathbf{b}\in \mathcal{A}_1}(w(\mathbf{b})-1)+\mid\mathcal{A}_2\mid+\sum_{\mathbf{b} \in \mathcal{A}_3}(w( b_{1},\ldots,b_{d})-1)}
\mathrm{i}^{\mid\mathcal{A}\mid} \nonumber \\
& & \cdot I_{n_1\cdots n_d} \otimes\left(\bigotimes_{j=d+1}^{d+r}(A_{K_2})^{\sum_{\mathbf{b} \in \mathcal{A}}b_j}\right).\label{Lemm10-3-5}
\end{eqnarray}

By our assumption, $\mathcal{A}$ contains at least one element whose last $r$ coordinates are not all equal to $0$, and at least one element whose first $d$ coordinates are not all equal to $0$. So (\ref{Lemm10-3-5}) can be restated as
\begin{equation}\label{LEq:10121}
H_\mathcal{A}\left(\frac{\pi}{2}\right)= \delta(\mathcal{A}_{1},\mathcal{A}_{2},\mathcal{A}_{3})\cdot I_{n_1\cdots n_d} \otimes\left(\bigotimes_{j=d+1}^{d+r}(A_{K_2})^{\sum_{\mathbf{b} \in \mathcal{A}}b_j}\right),
\end{equation}
where we set
\begin{eqnarray*}
\delta(\mathcal{A}_{1},\mathcal{A}_{2},\mathcal{A}_{3})=
\left\{\begin{array}{ll}
       (-1)^{\sum_{\mathbf{b}\in \mathcal{A}_1}(w(\mathbf{b})-1)+\mid\mathcal{A}_2\mid+\sum_{\mathbf{b} \in \mathcal{A}_3}(w( b_{1},\ldots,b_{d})-1)}
\mathrm{i}^{\mid\mathcal{A}\mid}, & \text{if $\mathcal{A}_{1}, \mathcal{A}_{2}, \mathcal{A}_{3}\not=\emptyset$};\\[0.1cm]
         (-1)^{\sum_{\mathbf{b}\in \mathcal{A}_1}(w(\mathbf{b})-1)+\mid\mathcal{A}_2\mid}
\mathrm{i}^{\mid\mathcal{A}\mid}, & \text{if $\mathcal{A}_{1}, \mathcal{A}_{2}\not=\emptyset$ and $\mathcal{A}_{3}=\emptyset$};\\[0.1cm]
   (-1)^{\sum_{\mathbf{b}\in \mathcal{A}_1}(w(\mathbf{b})-1)+\sum_{\mathbf{b} \in \mathcal{A}_3}(w( b_{1},\ldots,b_{d})-1)}
\mathrm{i}^{\mid\mathcal{A}\mid},  &   \text{if $\mathcal{A}_{1}, \mathcal{A}_{3}\not=\emptyset$ and $\mathcal{A}_{2}=\emptyset$};\\[0.1cm]
 (-1)^{\mid\mathcal{A}_2\mid+\sum_{\mathbf{b} \in \mathcal{A}_3}(w( b_{1},\ldots,b_{d})-1)}
\mathrm{i}^{\mid\mathcal{A}\mid},  &   \text{if $\mathcal{A}_{2}, \mathcal{A}_{3}\not=\emptyset$ and $\mathcal{A}_{1}=\emptyset$}; \\[0.1cm]
  (-1)^{\sum_{\mathbf{b} \in \mathcal{A}_3}(w( b_{1},\ldots,b_{d})-1)}
\mathrm{i}^{\mid\mathcal{A}\mid},     &    \text{if $\mathcal{A}_{3}\not=\emptyset$ and $\mathcal{A}_{1}=\mathcal{A}_{2}=\emptyset$}.
        \end{array}
\right.
\end{eqnarray*}

Therefore, if $(b_{d+1},\ldots,b_{d+r}) \neq \mathbf{0}$, then by (\ref{Lemm10-3-5}) and (\ref{LEq:10121}), $\NEPS(K_{n_1},\ldots,K_{n_{d}},r\odot K_2;\mathcal{A})$ has PST from vertex $\mathbf{u}$ to vertex $\mathbf{u}+(0,\ldots,0, b_{d+1},\ldots,b_{d+r})$ at time $\frac{\pi}{2}$, yielding (c.i).
On the other hand, if $(b_{d+1},\ldots,b_{d+r}) = \mathbf{0}$, then by (\ref{Lemm10-3-5})  and (\ref{LEq:10121}) we obtain (c.ii) immediately.
\qed
\end{Tproof}


\section{Concluding remarks}
\label{sec:cubelike}

Consider $\emptyset \ne \mathcal{A} \subseteq \mathbb{Z}_2^r \setminus \{\mathbf{0}\}$. It was proved in \cite[Theorem 1]{AB} that, if $\mathbf{c}(\mathcal{A}) \ne \mathbf{0}$, then the cubelike graph $\NEPS(r\odot K_2; \mathcal{A})$ admits PST from $\mathbf{u}$ to $\mathbf{u} + \mathbf{c}(\mathcal{A})$ at time $\frac{\pi}{2}$ for every vertex $\mathbf{u}$. A simple proof of this result was given in \cite[Theorem 2.3]{CC}, where it was also proved that if $\mathbf{c}(\mathcal{A}) = \mathbf{0}$ then $\NEPS(r\odot K_2; \mathcal{A})$ is periodic with period $\frac{\pi}{2}$. We now show that tools developed in previous sections can be used to give another proof of these two results. In addition, we obtain the transition matrix of $\NEPS(r\odot K_2; \mathcal{A})$ at time $\frac{\pi}{2}$.

\begin{theorem}
\label{cubelike1}
\emph{(\cite[Theorem 1]{AB}; \cite[Theorem 2.3]{CC})}
Let $\emptyset \ne \mathcal{A} \subseteq \mathbb{Z}_2^r \setminus \{\mathbf{0}\}$. Denote by $H_{\mathcal{A}}(t)$ the transition matrix of the cubelike graph $\NEPS(r\odot K_2; \mathcal{A})$ and set $\mathbf{c}(\mathcal{A}) = (c_1, \ldots, c_r)$.
\begin{itemize}
\item[\rm (a)] If $\mathbf{c}(\mathcal{A}) \neq \mathbf{0}$, then $\NEPS(r\odot K_2; \mathcal{A})$ admits PST from vertex $\mathbf{u}$ to vertex $\mathbf{u}+\mathbf{c}(\mathcal{A})$ at time $\frac{\pi}{2}$, for every vertex $\mathbf{u}$. Moreover,
$$
H_{\mathcal{A}}\left(\frac{\pi}{2}\right)
=(-\mathrm{i})^{\mid \mathcal{A}\mid}\bigotimes_{j=1}^{r}(A_{K_2})^{c_j}.
$$
\item[\rm (b)] If $\mathbf{c}(\mathcal{A})=\mathbf{0}$, then $\NEPS(r\odot K_2; \mathcal{A})$ is periodic with period $\frac{\pi}{2}$. Moreover,
$$
H_{\mathcal{A}}\left(\frac{\pi}{2}\right)
=(-\mathrm{i})^{\mid \mathcal{A}\mid}I_{2^r}.
$$
\end{itemize}
\end{theorem}

\begin{proof}
Denote by $H_{\mathbf{a}}(t)$ the transition matrix of $\NEPS(r\odot K_2; \{\mathbf{a}\})$. By Lemmas \ref{8} and \ref{Lemma2.5}(c), we have
\begin{eqnarray}
H_{\mathcal{A}}\left(\frac{\pi}{2}\right)&=&\prod_{\mathbf{a} \in \mathcal{A}} H_{\mathbf{a}}\left(\frac{\pi}{2}\right) \nonumber\\
&=&\prod_{\mathbf{a} \in \mathcal{A}}\left(-\mathrm{i}\bigotimes_{j=1}^{r}(A_{K_2})^{a_j}\right)\nonumber\\
&=&(-\mathrm{i})^{\mid \mathcal{A}\mid}\bigotimes_{j=1}^{r}(A_{K_2})^{\sum_{\mathbf{a} \in \mathcal{A}}a_j}. \nonumber
\end{eqnarray}
Thus, if $\mathbf{c}(\mathcal{A}) \neq \mathbf{0}$, then $\NEPS(r\odot K_2;\mathcal{A})$ has PST from $\mathbf{u}$ to $\mathbf{u}+\mathbf{c}(\mathcal{A})$ at time $\frac{\pi}{2}$, yielding (a). On the other hand, if $\mathbf{c}(\mathcal{A})=\mathbf{0}$, then $\NEPS(r\odot K_2;\mathcal{A})$ is periodic with period $\frac{\pi}{2}$, yielding (b).
\qed
\end{proof}


Finally, we would like to mention that the sufficient conditions in part (a) of Theorem \ref{Theorem4} and part (c.i) of Theorem \ref{10-3} are in general not necessary for $\NEPS$ of complete graphs to admit PST. We illustrate this by the following examples.

\begin{example}
\label{Exampe:2}
{\em
Consider the  graph $\NEPS(K_3,5\odot K_2;\mathcal{A})$, where
\begin{align*}
\mathcal{A}=&\{(0,0,0,0,0,1), (0,0,0,0,1,0), (0,0,0,0,1,1), (0,0,0,1,0,0), (0,0,0,1,0,1),(0,0,1,0,0,0),\\
 &~\, (0,0,1,0,0,1),(0,1,0,0,0,0),(0,1,0,0,0,1),(0,1,1,1,1,0),  (0,1,1,1,1,1)\}.
\end{align*}
Then $\mathbf{c}(\mathcal{A}) = \mathbf{0}$.
For any $\mathbf{a}=(a_1,a_2,a_3,a_4,a_5,a_6)\in \mathcal{A}$, define $\mathbf{a}^*=(a_2,a_3,a_4,a_5,a_6)$. Set
$$
\mathcal{A}_-^*=\{\mathbf{a}^*:\mathbf{a}\in\mathcal{A}\}.
$$
Let $A_{\mathbf{a}}$ and $A_{\mathbf{a}^*}$ be the adjacency matrices of $\NEPS(K_3,5\odot K_2;\{\mathbf{a}\})$ and $\NEPS(5\odot K_2;\{\mathbf{a}^*\})$, respectively. Lemma \ref{7} implies that $A_{\mathbf{a}}=I_3\otimes A_{\mathbf{a}^*}$.
Hence, by Lemmas \ref{3.1}(b) and \ref{8}, we have
\begin{eqnarray*}
H_\mathcal{A}(t) & = &\prod_{\mathbf{a} \in \mathcal{A}}H_{\mathbf{a}}(t)\\
   & = & \prod_{\mathbf{a} \in \mathcal{A}} I_3\otimes H_{\mathbf{a}^*}(t) \\
   & = &  I_3\otimes \prod_{\mathbf{a}^* \in \mathcal{A}_-^*} H_{\mathbf{a}^*}(t) \\
 & = &  I_3\otimes H_{\mathcal{A}_-^*}(t).
\end{eqnarray*}
It is known \cite{CC} that $\NEPS(5\odot K_2; \mathcal{A}_-^*)$ admits PST at time $\frac{\pi}{4}$. This implies that $\NEPS(K_3,5\odot K_2;\mathcal{A})$ admits PST at time $\frac{\pi}{4}$. On the other hand, we have $\mathbf{c}(\mathcal{A}) = \mathbf{0}$ as mentioned above. This shows that the sufficient condition in part (a) of Theorem \ref{Theorem4} is in general not necessary for $\NEPS(K_{n_1},\ldots,K_{n_d},r\odot K_2;\mathcal{A})$ to admit PST.
}
\qed\end{example}

\begin{example}
\label{Exampe:3}
{\em
Consider the  graph $\NEPS(K_4,5\odot K_2;\mathcal{A})$, where
\begin{align*}
\mathcal{A}=&\{(1,0,1,1,1,1),(1,1,0,1,1,1), (1,1,1,0,1,1),(1,1,1,1,0,1),(1,0,0,1,1,1),
(1,0,1,0,1,1), \\
&~~(1,0,1,1,0,1),(1,0,1,1,1,0),(1,1,0,0,1,1),(1,1,0,1,0,1),(1,1,0,1,1,0),(1,1,1,0,0,1),\\
&~~(1,1,1,0,1,0),(1,1,1,1,0,0),(1,0,1,1,0,0), (1,0,0,1,1,0),(1,1,0,1,0,0),(1,1,0,0,1,0),\\
&~~(1,0,1,0,1,0),(1,1,1,0,0,0)\}.
\end{align*}
For any $\mathbf{a}=(a_1,a_2,a_3,a_4,a_5,a_6)\in \mathcal{A}$, define $\mathbf{a}^*=(a_2,a_3,a_4,a_5,a_6)$. Let
$$
\mathcal{A}_-^*=\{\mathbf{a}^*:\mathbf{a}\in\mathcal{A}\}.
$$Let $A_{\mathbf{a}}$ and $A_{\mathbf{a}^*}$ be the adjacency matrices of $\NEPS(K_4,5\odot K_2;\{\mathbf{a}\})$  and $\NEPS(5\odot K_2;\{\mathbf{a}^*\})$,  respectively.  Lemma \ref{7} implies that  $A_{\mathbf{a}}=A_{K_4}\otimes A_{\mathbf{a}^*}$.
According to Lemma \ref{alg} and Equation (\ref{H_ADec3}), we have
$$H_{\mathbf{a}}(t)=H_{A_{K_4\otimes A_{\mathbf{a}^*}}}(t)=E_3 \otimes H_{\mathbf{a}^*}(3t)+E_{-1} \otimes H_{\mathbf{a}^*}(-t),$$
where
\begin{equation*}
E_{3}=\frac{1}{4}\left(
  \begin{array}{rrrrr}
    1 & 1 & 1 &  1\\
    1 & 1 & 1 &  1\\
       1 & 1 & 1 &  1\\
       1 & 1 & 1 &  1
  \end{array}
\right),\; E_{-1}=\frac{1}{4}\left(
  \begin{array}{rrrrr}
    3 & -1 & -1 &   -1\\
    -1 & 3 & -1 &   -1\\
   -1 & -1 & 3 &  -1\\
    -1 & -1 &  -1 & 3
  \end{array}
\right).
\end{equation*}
Hence, by Lemmas \ref{3.1}(b) and \ref{8}, we have
\begin{eqnarray}\label{equation56}
H_\mathcal{A}(t) & = &\prod_{\mathbf{a} \in \mathcal{A}}H_{\mathbf{a}}(t)\nonumber\\
   & = & \prod_{\mathbf{a} \in \mathcal{A}}\left( E_3 \otimes H_{\mathbf{a}^*}(3t)+E_{-1} \otimes H_{\mathbf{a}^*}(-t)\right) \nonumber\\
   & = &  E_3 \otimes \prod_{\mathbf{a}^* \in \mathcal{A}_-^*}H_{\mathbf{a}^*}(3t)+E_{-1} \otimes \prod_{\mathbf{a}^* \in \mathcal{A}_-^*}H_{\mathbf{a}^*}(-t) \nonumber\\
 & = &  E_3 \otimes H_{\mathcal{A}_-^*}(3t)+E_{-1} \otimes H_{\mathcal{A}_-^*}(-t).
\end{eqnarray}
By Lemma \ref{Lemma2.5}(c), we have
$$H_{\mathbf{a}^*}\left(\frac{\pi}{2}\right) =-\mathrm{i}\bigotimes_{j=1}^{5}(A_{K_2})^{a_j}.$$
Note that $|\mathcal{A}_-^*|=20$ and $\sum_{\mathbf{a}^* \in \mathcal{A}_-^*}\mathbf{a}^*=(0,0,0,0,0)$.
Thus, by Lemma \ref{8}, we have
\begin{eqnarray*}
H_{\mathcal{A}_-^*}\left(\frac{\pi}{2}\right)&=&\prod_{\mathbf{a}^* \in \mathcal{A}_-^*}H_{\mathbf{a}^*}\left(\frac{\pi}{2}\right) \\
&=&(-\mathrm{i})^{|\mathcal{A}_-^*|}
\prod_{\mathbf{a}^* \in \mathcal{A}_-^*}\left(\bigotimes_{j=1}^{5}(A_{K_2})^{a_j}\right)\\
&=&\bigotimes_{j=1}^{5}(A_{K_2})^{\sum_{\mathbf{a}^* \in \mathcal{A}_-^*}a_j} \\
&=&I_{2^5},
\end{eqnarray*}
yielding $H_{\mathcal{A}_-^*}\left(\pi\right)=I_{2^5}$.
Thus, when $t=\frac{\pi}{4}$, Equation (\ref{equation56}) gives 
\begin{eqnarray}\label{equation78}
H_\mathcal{A}\left(\frac{\pi}{4}\right)
 & = &  E_3 \otimes H_{\mathcal{A}_-^*}\left(\frac{3\pi}{4}\right)+E_{-1} \otimes H_{\mathcal{A}_-^*}\left(-\frac{\pi}{4}\right)\nonumber\\
  & = &  E_3 \otimes H_{\mathcal{A}_-^*}\left(\frac{3\pi}{4}\right)+E_{-1} \otimes \left( H_{\mathcal{A}_-^*}\left(\frac{3\pi}{4}\right) H_{\mathcal{A}_-^*}(-\pi)\right)\nonumber\\
  &=& E_3 \otimes H_{\mathcal{A}_-^*}\left(\frac{3\pi}{4}\right)+E_{-1} \otimes H_{\mathcal{A}_-^*}\left(\frac{3\pi}{4}\right)\nonumber\\
  &=& I_4\otimes H_{\mathcal{A}_-^*}\left(\frac{3\pi}{4}\right).
 \end{eqnarray}
It is known \cite{CC} that $\NEPS(5\odot K_2; \mathcal{A}_-^*)$ admits PST at time $\frac{3\pi}{4}$. According to (\ref{equation78}), $\NEPS(K_4,5\odot K_2;\mathcal{A})$ admits PST at time $\frac{\pi}{4}$. On the other hand, we have $\mathbf{c}(\mathcal{A}_-^*) = \mathbf{0}$. This shows that the sufficient condition in part (c.i) of Theorem \ref{10-3} is in general not necessary for $\NEPS(K_{n_1},\ldots,K_{n_d},r\odot K_2;\mathcal{A})$ to admit PST.
}
\qed\end{example}

\bigskip

\noindent \textbf{Acknowledgement}~~The authors greatly appreciate the anonymous referees for their comments and suggestions.

\end{document}